\documentclass[a4paper,10pt]{article}
\usepackage[margin=1in]{geometry}
\usepackage[utf8]{inputenc}
\usepackage[T1]{fontenc}
\usepackage{lmodern}
\usepackage{microtype}
\usepackage{authblk}
\usepackage{amsmath,amsthm,amssymb,amsfonts}
\usepackage{algorithm,algorithmic}
\usepackage{graphicx}
\usepackage{booktabs,multirow,makecell}
\usepackage{pgfplots}
\pgfplotsset{compat=1.18}
\usepgfplotslibrary{groupplots,fillbetween}
\usepackage{placeins}
\newcommand{\arxivexpandedexperiments}{}
\usepackage{xcolor}
\usepackage{natbib}
\usepackage{hyperref}
\hypersetup{colorlinks=true,linkcolor=blue,filecolor=magenta,urlcolor=cyan}
\usepackage{cleveref}

\newcommand{\R}{\mathbb R}
\newcommand{\E}{\mathbb E}
\newcommand{\B}{\mathcal B}
\newcommand{\C}{\mathcal C}
\newcommand{\Pset}{\mathcal P}
\newcommand{\norm}[1]{\left\|#1\right\|}
\newcommand{\ip}[2]{\left\langle#1,#2\right\rangle}
\newcommand{\dist}{\operatorname{dist}}
\newcommand{\argmin}{\operatorname*{arg\,min}}
\newcommand{\ind}{\delta}

\theoremstyle{plain}
\newtheorem{theorem}{Theorem}
\newtheorem{lemma}{Lemma}

\newtheorem{assumption}{Assumption}
\newtheorem{remark}{Remark}

\title{\bf  Constrained Nonconvex Stochastic Optimization with One Projection}
\date{}

\author[1]{Yuyang Deng}
\author[2]{Mohammadreza M. Kalan}
\author[1]{Eitan J. Neugut}
\author[3]{Mehrdad Mahdavi}
\affil[1]{Columbia University, Department of Statistics}
\affil[2]{Univ Rennes, Ensai, CNRS, CREST--UMR 9194, F-35000 Rennes, France}
\affil[3]{Pennsylvania State University, Department of Computer Science and Engineering}
\affil[ ]{\texttt{yd2824@columbia.edu}\quad
  \texttt{mohammadreza.kalan@ensai.fr}\quad
  \texttt{eitan.neugut@columbia.edu}\quad\texttt{mzm616@psu.edu}}
\begin{document}

\maketitle

\begin{abstract}
Constrained nonconvex optimization has seen increasing application in modern machine learning, such as safe LLM alignment/finetuning, transfer learning and rank-constrained continual learning. The commonly used approach is Projected SGD which requires projection at every iteration. However, projection onto a functional constraint can cost substantially more than a
stochastic first-order update. We study whether this operation can be deferred
until the end of nonconvex stochastic optimization, and only do it once. For weakly convex,
possibly nonsmooth objectives with regular convex constraints, we give a
penalized proximal method that uses $\widetilde O(\epsilon^{-4})$ stochastic
subgradients and one projection onto the constraint set. The returned point
is exactly feasible and has a small Moreau-envelope stationarity measure;
for smooth objectives, the same algorithm controls the projected-gradient
mapping. For smooth nonconvex constraints, we assume a global lower bound on the
infeasible constraint slope which permits arbitrary, possibly infeasible
initialization. An exact-penalty variant attains the same stochastic-oracle
order and returns an exactly feasible point within $O(\epsilon)$ of an
$\epsilon$-KKT point. All intermediate updates use only projections onto a
simple Euclidean ball. These are oracle-complexity guarantees: the cost of
the single terminal projection is separate, and no efficient projection
algorithm for a general nonconvex set is assumed to follow from regularity.
\end{abstract}

\section{Introduction}
We consider stochastic optimization with a functional constraint,
\begin{equation}
    \min_{x\in\R^d} f(x)
    \qquad\text{subject to}\qquad g(x)\leq0,
    \qquad \C:=\{x:g(x)\leq0\}.
    \label{eq:problem}
\end{equation}
The objective $f$ may be nonconvex and nonsmooth, and the constraint function
$g$ may be convex or nonconvex. In modern machine learning, this problem formulation is widely used in many application scenarios, such as transfer learning~\citep{NEURIPS2025_85d4d3ba,kalan2026neymanpearson}, constrained LLM alignment~\citep{zhang2025alignment,peng2025enhancing} and rank-constrained LLM finetuning~\citep{hu2022lora,zhang2023adalora,jang2024lora,park2025riemannian}. Nowadays, as the data amount and the number of parameters to optimize increase dramatically, a computationally efficient algorithm is needed to deploy large-scale model training.

Standard constrained stochastic optimization relies on projecting onto the feasible set after each update step. This method can work well for the simple constraint set such as $l_2$ unit ball or simplex, since the projection onto them can be computed efficiently.  However, this approach becomes a severe bottleneck under general functional constraints as projection is itself an expensive optimization problem and repeatedly solving that sub-problem can dominate the total computational cost.  This motivates a fundamental question:  

\emph{Can we retain a nonconvex stationarity rate while projecting onto the
constraint set only once, at the end?}

For convex objectives and constraints, one-projection methods already show
that the strict feasibility need not be maintained during optimization
\citep{mahdavi2012one,yang2017richer}. Extending this principle to nonconvex
objectives requires a different conclusion and a different argument. A small
objective gap is no longer the appropriate target, averaging outer iterates
need not preserve stationarity, and projecting an arbitrary nearly stationary
point of a penalty objective need not preserve the desired certificate.
With nonconvex constraints, there is a further obstruction: the
constraint violation can have stationary points outside the feasible set.

To bridge this gap, we propose a framework that separates stochastic
progress from the final feasibility correction. Our algorithm approximately
solves a sequence of strongly convex penalized proximal subproblems over a
simple bounding ball, allowing intermediate iterates to remain infeasible.
After selecting one outer iterate at random, we apply a single exact
projection onto $\C$. The analysis identifies conditions under which this
final correction yields an exactly feasible output with a stationarity
guarantee for the original constrained problem.

\paragraph{Our results.}
We introduce the Penalized Proximal Subgradient Method with One Projection
and analyze it for weakly convex, possibly nonsmooth objectives under two
constraint settings. Under the assumptions specified for each setting, our
stationarity guarantees hold with probability at least $2/3$ and require
$\widetilde O(\epsilon^{-4})$ stochastic subgradient calls, for fixed
problem parameters.
\begin{itemize}
\setlength{\itemsep}{2pt}
    \item \emph{Convex constraints.} Starting from a supplied feasible
    point, the algorithm approximately solves softplus-penalized proximal
    subproblems and returns an exactly feasible output. For smooth
    objectives, Theorem~\ref{thm:convex-smooth} guarantees an
    $\epsilon$-stationary point measured by the projected-gradient mapping.
    For nonsmooth objectives, Theorem~\ref{thm:convex-nonsmooth} bounds the
    Moreau-envelope residual by $\epsilon$. In both cases, the oracle
    complexity matches, up to logarithmic factors, the standard
    $\epsilon^{-4}$ stochastic-oracle complexity for weakly convex
    optimization~\citep{davis2019guided,davis2019model}.

    \item \emph{Nonconvex constraints.} For smooth nonconvex $g$, we use
    an exact hinge penalty and impose a global infeasible-slope condition
    to rule out stationary points of the constraint violation outside the
    feasible set. This condition is inspired by boundary gradient
    lower bounds in convex constrained
    optimization~\citep{mahdavi2012one,yang2017richer}, but is stronger:
    it controls the slope throughout the ball's infeasible region and
    accounts for directions blocked by the ball's boundary. Under this condition and the
    remaining regularity assumptions, the algorithm allows any initial
    point in the ball and enforces no intermediate feasibility.
    Theorem~\ref{thm:nonconvex} shows that one final projection returns an
    exactly feasible output within $O(\epsilon)$ of an $\epsilon$-KKT
    point of \eqref{eq:problem}, with the same
    $\widetilde O(\epsilon^{-4})$ stochastic subgradient complexity.
\end{itemize}

\section{Related Work}
\paragraph{Reducing projections.}
\citet{mahdavi2012one} introduced stochastic methods with a single final
projection for convex optimization. \citet{yang2017richer} developed a
broader penalty-based theory of reduced projections and improved rates under
constraint regularity. Our convex-constraint analysis uses this penalty
principle inside proximal subproblems, but the outer target is stationarity
of a nonconvex objective. Conditional-gradient methods provide a different
way to avoid projections by using a linear minimization oracle
\citep{jaggi2013frankwolfe}. That oracle can itself be costly for a functional
constraint; we instead use values and gradients of $g$ during the iterations.

\paragraph{Weak convexity and constrained stationarity.}
Proximally guided stochastic subgradient methods solve weakly convex
problems through strongly convex subproblems \citep{davis2019guided}.
Moreau-envelope stationarity and the $\epsilon^{-4}$ stochastic-oracle order
are established tools in this setting \citep{davis2019model}.
Our outer--inner organization builds on these ideas. For nonconvex feasible
sets, \citet{davis2025sets} analyze stochastic model-based methods over
proximally smooth sets, including set approximations followed by retractions
that restore feasibility each iteration. Their geometric framework is
broader than our single smooth inequality. Our nonconvex result instead
permits infeasible iterates throughout, at the price of a global slope
condition and one terminal projection oracle.

\paragraph{Functional constraints and exact penalties.}
\citet{ma2020quadratic} and \citet{boob2019functional} regularize both
objectives and constraints and solve convex constrained subproblems.
The IQRC method of \citet{ma2020quadratic} assumes an approximately feasible
initial point and allows approximate feasibility; it is not a method defined
by one final exact projection. Its fully stochastic constraint-oracle model
also differs from our exact access to $g$ and $\nabla g$.
\citet{jia2022firstorder} study deterministic weakly convex functional
constraints, including feasible approximate stationary solutions and
Fritz--John guarantees without constraint qualification.  More directly, \citet{liu2025spider} connect
exact-penalty stationarity to constrained KKT conditions under a uniform
Slater-type constraint qualification and give a stochastic method for expectation
constraints. We do not claim a new exact-penalty principle. We use a related
slope argument, accounting also for the bounding set's normal cone, to
justify a single terminal feasibility correction and its effect on the
reported output.

\section{Main Results}
\label{sec:results}
\noindent\textbf{Notations.}~All norms are Euclidean. Write $[u]_+:=\max\{u,0\}$ and let
$\Pset_D(x):=\argmin_{y\in D}\norm{y-x}$ denote the Euclidean projection,
possibly set-valued. The symbols $\partial f$ and $N_D=\partial\ind_D$
denote the limiting subdifferential and normal cone; $\ind_D$ is the extended
real-valued indicator of $D$. A function is $\ell$-weakly convex if adding
$\ell\norm{\cdot}^2/2$ makes it convex. In particular, an $\ell$-smooth
function is $\ell$-weakly convex.

\subsection{Penalized Proximal Subgradient Method with One Projection}
\label{sec:algorithm}
Our method reduces the original nonconvex problem to a sequence of
\emph{strongly convex subproblems over a simple ball}. At outer iteration
$t$, we form the penalized proximal objective
\begin{equation}
    H_t(z):=f(z)+p(g(z))+\frac{\beta}{2}\norm{z-x_t}^2,
    \qquad \min_{z\in\B} H_t(z),
    \label{eq:generic-subproblem}
\end{equation}
where $\B=\{z:\norm{z}\leq1\}$ and $p$ penalizes constraint violation.
The quadratic term is chosen to dominate the negative curvature of both
$f$ and $p\circ g$, so that the \emph{entire objective $H_t$} is
$m$-strongly convex on $\B$, with $m>0$.
For convex $g$, the softplus penalty is convex, and $\beta=2\ell$ gives
$m=\ell$. For nonconvex $g$, we use the hinge penalty
$p(u)=\lambda[u]_+$ and choose $\beta>\ell+\lambda L_g$, giving
$m=\beta-\ell-\lambda L_g$. The following subsections specify the penalty
parameters and assumptions for each setting.

Each subproblem is solved approximately by stochastic subgradient steps
projected onto $\B$. Their average becomes the next outer iterate
$x_{t+1}$. Strong convexity lets us choose a fixed inner budget $K$ to
obtain a prescribed function error
\begin{equation}
    H_t(x_{t+1})-\min_{z\in\B}H_t(z)\leq\epsilon_{\rm in}
    \label{eq:generic-inner-gap}
\end{equation}
with the probability required by the analysis. The algorithm does not
need to evaluate the minimum or test this gap during its execution.

\begin{algorithm}[t!]
\caption{Penalized Proximal Subgradient Method with One Projection}
\label{alg:one-projection}
\begin{algorithmic}[1]
\REQUIRE $x_0\in\B$, $T,K\geq2$, penalty $p$,
index set $\mathcal I$, and $\beta,m>0$ such that $H_t$ is $m$-strongly convex on $\B$.
\FOR{$t=0,\ldots,T-1$}
    \STATE Form $H_t(z)=f(z)+p(g(z))+\frac{\beta}{2}\norm{z-x_t}^2$.
    \STATE $x_{t+1}\gets\operatorname{PSM}(H_t,x_t)$ \hfill \textcolor{gray}{\# approximately minimize $H_t$ over $\B$}.
\ENDFOR
\STATE Sample $J$ uniformly from the index set $\mathcal I$.
\STATE \textbf{Return} $\hat x\in\Pset_{\C}(x_J)$.
\STATE \textbf{Subroutine} $\operatorname{PSM}(H,x)$: 
\STATE \hspace{1em} Initialize $z_1\gets x$.
\FOR{$s=1,\ldots,K$}
    \STATE Set $\eta_s\gets 2/[m(s+1)]$ and query $\widetilde v(z_s;\xi_s)$ with fresh $\xi_s$.
    \STATE Choose $h_s\in\widetilde v(z_s;\xi_s)+\beta(z_s-x)+\partial(p\circ g)(z_s)$.
    \STATE $z_{s+1}\gets\Pset_{\B}(z_s-\eta_s h_s)$.
\ENDFOR
\STATE \hspace{1em} \textbf{Return} $K^{-1}\sum_{s=1}^K z_s$.
\end{algorithmic}
\end{algorithm}

Both loops operate in $\B$ without enforcing feasibility in $\C$.
After the outer loop, we select an iterate at random and project it onto
$\C$ once. For convex constraints, the analysis bounds the infeasibility
of each outer iterate. For nonconvex constraints, it instead shows that
a selected iterate with a small proximal residual has a nearby feasible,
approximately stationary witness. Thus intermediate iterates may be infeasible.
Algorithm~\ref{alg:one-projection} describes the pseudocode.

For softplus, the penalty subgradient is simply
$\nabla(p_\gamma\circ g)(z_s)$; for hinge, use the selection in
Section~\ref{sec:nonconvex}. Together with \eqref{eq:oracle}, this makes
$h_s$ an unbiased stochastic subgradient of $H$, even when $f$ is nonsmooth.
The inner steps project only onto $\B$; the sole projection onto $\C$
appears in the outer algorithm's final return step.

\subsection{Nonconvex Objective with a Convex Constraint Set}
\label{sec:convex}

\begin{assumption}[Shared assumptions for convex constraints]
\label{ass:convex}
\leavevmode
\begin{itemize}
\setlength{\itemsep}{2pt}
    \item The safeguard set is $\B=\{x:\norm{x}\leq1\}$. The feasible set
    $\C$ is nonempty, closed, and contained in $\B$. A point $x_0\in\C$
    is supplied.
    \item The function $f$ is $\ell$-weakly convex and $G_f$-Lipschitz on
    an open convex neighborhood of $\B$, where $\ell,G_f>0$.
    With fresh randomness at each call, the oracle satisfies
    \begin{equation}
        \E[\widetilde v(x;\xi)\mid\mathcal F]
        =v(x)\in\partial f(x),
        \qquad \norm{\widetilde v(x;\xi)}\leq G_f\quad\text{a.s.}
        \label{eq:oracle}
    \end{equation}
    Here $x$ is measurable with respect to the past history $\mathcal F$.
    \item The function $g$ is convex and continuously differentiable on
    that neighborhood. For constants $G_g,\sigma>0$,
    \[
        \norm{\nabla g(x)}\leq G_g\quad(x\in\B),
        \qquad
        \norm{\nabla g(z)}\geq\sigma\quad(z\in\B,\ g(z)=0).
    \]
    Both $g(x)$ and $\nabla g(x)$ can be evaluated exactly.
\end{itemize}
\end{assumption}

The unit radius is a normalization; a known bounding ball can be translated
and rescaled. The boundary condition implies a global error bound on $\B$,
$\dist(x,\C)\leq[g(x)]_+/\sigma$, proved in
Lemma~\ref{lem:convex-eb}. It is a property of the chosen defining function,
not merely of the feasible set.

\paragraph{Algorithm.}
At outer iteration $t$, define
\begin{equation}
    F_t(x):=f(x)+\ell\norm{x-x_t}^2,
    \qquad
    p_\gamma(u):=\gamma\log\bigl(1+\exp(\lambda_0u/\gamma)\bigr).
    \label{eq:convex-penalty}
\end{equation}
With $p=p_\gamma$ and $\beta=2\ell$, the subproblem in
\eqref{eq:generic-subproblem} is $H_t=F_t+p_\gamma\circ g$. Since
$F_t$ is $\ell$-strongly convex and $p_\gamma\circ g$ is convex,
$H_t$ is $\ell$-strongly convex on $\B$. Set
\begin{equation}
    \lambda_0:=\frac{2(G_f+4\ell)}{\sigma},\qquad
    \gamma:=\frac1K.
    \label{eq:convex-settings}
\end{equation}
Use Algorithm~\ref{alg:one-projection} with
$p=p_\gamma$, $\beta=2\ell$, $m=\ell$, and
$\mathcal I=\{1,\ldots,T\}$.

Set $\Delta_f:=f(x_0)-\min_{x\in\B}f(x)$ throughout this subsection.
We state the two objective regimes separately because their stationarity
certificates have different meanings.

\paragraph{Smooth objectives.}
\begin{assumption}[Smooth objective]
\label{ass:convex-smooth}
\leavevmode
  The function $f$ is differentiable with an $\ell$-Lipschitz
    gradient on the same open convex neighborhood of $\B$.
 
\end{assumption}

We use the projected-gradient mapping norm
\begin{equation}
    \mathcal G_\ell(x)
    :=\ell\norm{x-\Pset_{\C}(x-\nabla f(x)/\ell)}.
    \label{eq:gradient-mapping}
\end{equation}
At a feasible point this measures the displacement of a projected gradient
step. It vanishes exactly when
$0\in\nabla f(x)+N_{\C}(x)$, the constrained first-order condition.
A small value therefore bounds the change made by a projected gradient step
from the returned point itself.

\begin{theorem}[\textbf{Smooth objective with convex constraints}]
\label{thm:convex-smooth}
Suppose Assumptions~\ref{ass:convex} and \ref{ass:convex-smooth} hold. For any
$0<\epsilon\leq1$, run Algorithm~\ref{alg:one-projection} with
$p=p_\gamma$, $\beta=2\ell$, $m=\ell$, the penalty parameters in
\eqref{eq:convex-settings}, and $\mathcal I=\{1,\ldots,T\}$.
There are choices
\begin{equation}
    T=O\!\left(1+\frac{\ell\Delta_f}{\epsilon^2}\right),
    \qquad K=\widetilde O(\epsilon^{-2}),
    \label{eq:convex-smooth-rate}
\end{equation}
specified in Appendix~\ref{app:convex-schedule}, for which the output
satisfies
\[
    \hat x\in\C,
    \qquad
    \Pr\{\mathcal G_\ell(\hat x)\leq\epsilon\}\geq\frac23.
\]
For fixed problem parameters, the algorithm uses
$TK=\widetilde O(\epsilon^{-4})$ stochastic gradient calls and exactly
one projection onto $\C$.
\end{theorem}

\begin{remark}
The $\widetilde O(\epsilon^{-4})$ rate matches the optimal
unconstrained smooth nonconvex exponent under the oracle models of
\citet{arjevani2023lower,jin2026bounded}, while using one final projection.
This is a comparison of complexity exponents; a matching lower bound under
all assumptions and oracle access of Theorem~\ref{thm:convex-smooth}
is not established here.
\end{remark}

\paragraph{Nonsmooth objectives.}
Under Assumption~\ref{ass:convex}, $f$ may be nonsmooth, so the gradient
mapping need not be defined. Instead, we follow the standard convergence measure in nonsmooth nonconvex optimization~\cite{davis2019model} and define the Moreau envelope and its
residual by
\begin{equation}
\begin{split}
    \Phi_\ell(x)&:=\min_{y\in\C}
         \{f(y)+\ell\norm{y-x}^2\},\qquad
    Q_\ell(x):=\argmin_{y\in\C}\{f(y)+\ell\norm{y-x}^2\},\\
    \mathcal S_\ell(x)&:=\norm{\nabla\Phi_\ell(x)}
         =2\ell\norm{x-Q_\ell(x)}.
\end{split}
\label{eq:moreau-measure}
\end{equation}
The proximal point $y=Q_\ell(x)$ is unique. Its optimality condition gives
\begin{equation}
    \norm{x-y}=\frac{\mathcal S_\ell(x)}{2\ell},
    \qquad
    \dist(0,\partial f(y)+N_{\C}(y))\leq\mathcal S_\ell(x).
    \label{eq:moreau-intuition}
\end{equation}
Thus a small residual means that $x$ is close to a feasible point with a
small first-order residual. This nearby-point interpretation is useful for
nonsmooth functions, whose subgradients need not vary continuously. It does
not assert a small subgradient residual at the returned point itself.

\begin{theorem}[\textbf{Nonsmooth objective with convex constraints}]
\label{thm:convex-nonsmooth}
Suppose Assumption~\ref{ass:convex} holds. For any $0<\epsilon\leq1$,
run Algorithm~\ref{alg:one-projection} with
$p=p_\gamma$, $\beta=2\ell$, $m=\ell$, the penalty parameters in
\eqref{eq:convex-settings}, and $\mathcal I=\{1,\ldots,T\}$.
There are choices
\begin{equation}
    T=O\!\left(1+\frac{\ell\Delta_f}{\epsilon^2}\right),
    \qquad K=\widetilde O(\epsilon^{-2}),
    \label{eq:convex-nonsmooth-rate}
\end{equation}
specified in Appendix~\ref{app:convex-schedule}, for which the output
satisfies
\[
    \hat x\in\C,
    \qquad
    \Pr\{\mathcal S_\ell(\hat x)\leq\epsilon\}\geq\frac23.
\]
For fixed problem parameters, the algorithm uses
$TK=\widetilde O(\epsilon^{-4})$ stochastic subgradient calls and exactly
one projection onto $\C$.
\end{theorem}

\begin{remark}
The rate matches the standard $O(\epsilon^{-4})$ Moreau-envelope
stationarity bound for weakly convex stochastic optimization
\citep{davis2019guided,davis2019model}, up to logarithms, with one final
projection. In the smooth unconstrained case, a small Moreau residual also
bounds the gradient norm up to a constant, connecting this measure to the
smooth lower-bound benchmark. The same oracle-model qualification applies.
\end{remark}

\subsection{Nonconvex Objective with a Nonconvex Constraint Set}
\label{sec:nonconvex}

We now allow $g$ to be nonconvex and $x_0$ to be infeasible. The following
condition excludes infeasible stationary points of the constraint violation,
including points where the boundary of $\B$ blocks further progress.

\begin{assumption}[Nonconvex constraint and infeasible initialization]
\label{ass:nonconvex}
\leavevmode
\begin{itemize}
\setlength{\itemsep}{2pt}
    \item The second item of Assumption~\ref{ass:convex} holds for the
    objective and its oracle. The set $\C$ is nonempty, closed, and
    contained in $\operatorname{int}\B$. The initial point $x_0\in\B$
    is arbitrary.
    \item The function $g$ is continuously differentiable with an
    $L_g$-Lipschitz gradient on an open convex neighborhood of $\B$,
    where $L_g>0$. Its values and gradients are available exactly.
    Write $G_g:=\max_{x\in\B}\norm{\nabla g(x)}$.
    \item There are $\sigma_{\rm bd},\sigma_{\rm out}>0$ such that
    \begin{align}
        \norm{\nabla g(z)}&\geq\sigma_{\rm bd}
           &&(z\in\B,\ g(z)=0),\label{eq:boundary-cq}\\
        \dist(0,\nabla g(x)+N_{\B}(x))&\geq\sigma_{\rm out}
           &&(x\in\B,\ g(x)>0).\label{eq:infeasible-slope}
    \end{align}
\end{itemize}
\end{assumption}

Condition \eqref{eq:infeasible-slope} is a global sufficient condition for
this guarantee, not a consequence of boundary regularity. It says that, within $\B$, as long as the constraint is not satisfied, one can always find a direction to decrease the constraint violation. It is stronger
than simply requiring $\nabla g\ne0$ on $g=0$. Appendix~\ref{app:obstruction}
gives a smooth counterexample to removing it from the present argument.
Meanwhile, \eqref{eq:boundary-cq} and smoothness imply that $\C$ has a
unique projection in the tube of radius
$R:=\sigma_{\rm bd}/L_g$; no separate geometric assumption is needed
(Lemma~\ref{lem:geometry}).

\paragraph{Algorithm: the same subroutine with an exact penalty.}
Consider $\Psi_\lambda(x):=f(x)+\lambda[g(x)]_+$ on $\B$ and choose
\begin{equation}
    \lambda>\frac{G_f}{\sigma_{\rm out}},\qquad
    \beta>\ell+\lambda L_g,\qquad
    m:=\beta-\ell-\lambda L_g.
    \label{eq:nc-parameters}
\end{equation}
Run Algorithm~\ref{alg:one-projection} with $p(u)=\lambda[u]_+$,
these $\beta,m$, and
\begin{equation}
    \mathcal I=\{0,\ldots,T-1\}.
    \label{eq:nc-settings}
\end{equation}
For the composite penalty, select $\lambda\nabla g(z)$ when $g(z)>0$
and $0$ when $g(z)\leq0$. At $g(z)=0$, zero is a valid subgradient.
Here $p\circ g$ is $\lambda L_g$-weakly convex, so the inner objective
$H_t(z)=\Psi_\lambda(z)+\frac\beta2\norm{z-x_t}^2$ is $m$-strongly
convex over the ball. The parameter choices are sequential and compatible:
choose $\lambda$ first and then $\beta$. Neither the inner nor the outer
iterates are required to be feasible, and there is no feasibility-restoration
phase before the final line of the algorithm.

\paragraph{Convergence measure.}
Our goal is an exactly feasible output $\hat x$ that is close to a
near-stationary point of the \emph{original} constrained problem:
\begin{equation}
    \exists y\in\C:\quad
    \norm{\hat x-y}\leq\delta,
    \qquad \dist(0,\partial f(y)+N_{\C}(y))\leq\epsilon.
    \label{eq:near-stationarity}
\end{equation}
The two tolerances distinguish proximity from first-order accuracy. This is
the same nearby-point interpretation underlying
\eqref{eq:moreau-intuition}; it does not assert a small nonsmooth KKT
residual at $\hat x$ itself.
To prove it, we track the auxiliary penalty residual
\begin{equation}
    Q_{\lambda,\beta}(x):=\argmin_{y\in\B}
       \left\{\Psi_\lambda(y)+\frac\beta2\norm{y-x}^2\right\},
    \qquad
    \mathcal R_{\lambda,\beta}(x):=
       \beta\norm{x-Q_{\lambda,\beta}(x)}.
    \label{eq:nc-residual}
\end{equation}
The key is not merely to make a penalized problem stationary. Once this
residual is below the positive margin
$h_{\rm feas}:=\lambda\sigma_{\rm out}-G_f$, its proximal witness
$y=Q_{\lambda,\beta}(x)$ must be feasible. Its optimality condition then
certifies original-problem stationarity, and projecting $x$ changes its
distance to that witness by at most a factor of two.

\begin{theorem}[\textbf{Nonsmooth objective with nonconvex constraints}]
\label{thm:nonconvex}
Suppose Assumption~\ref{ass:nonconvex} holds and use
\eqref{eq:nc-parameters}--\eqref{eq:nc-settings}. Let
$\Delta_\lambda:=\Psi_\lambda(x_0)-\min_{x\in\B}\Psi_\lambda(x)$.
For $0<\epsilon<\min\{h_{\rm feas},\beta R,1\}$, there are choices
\begin{equation}
    T=O\!\left(1+\frac{\beta\Delta_\lambda}{\epsilon^2}\right),
    \qquad K=\widetilde O(\epsilon^{-2}),
    \label{eq:nc-rate}
\end{equation}
specified in Appendix~\ref{app:nc-schedule}, with the following guarantee.
The output is exactly feasible. With probability at least $2/3$, the final
projection is unique, $\mathcal R_{\lambda,\beta}(x_J)\leq\epsilon$, and
there exist $y\in\C$, $v\in\partial f(y)$, and $\mu\in[0,\lambda]$ such that
\begin{equation}
    \norm{\hat x-y}\leq\frac{2\epsilon}{\beta},\qquad
    \norm{v+\mu\nabla g(y)}\leq\epsilon,\qquad \mu g(y)=0.
    \label{eq:nc-guarantee}
\end{equation}
For fixed problem parameters, the algorithm uses
$TK=\widetilde O(\epsilon^{-4})$ stochastic subgradients and exactly one
projection onto $\C$.
\end{theorem}


\begin{remark}
Here we also achieve the standard
$\epsilon^{-4}$ rate.
The stochastic-oracle exponent is unchanged, although its constants now
depend on the infeasible-slope margin and the penalty's weak convexity.
Unlike the convex-constraint results above, this theorem does not require a feasible initial
point. It uses no projections onto local feasible-set models and does not
keep the intermediate iterates feasible. It does, however, assume an oracle
that returns an \emph{exact nearest point} in the last line. Compactness
ensures such a point exists on every outcome; uniqueness is guaranteed on
the stated success event. Local uniqueness is not a computational algorithm
for finding that point. Thus our bound counts one call to a potentially
expensive projection oracle and does not claim polynomial-time nonconvex projection.
For smooth $f$, replace $v$ in \eqref{eq:nc-guarantee} by $\nabla f(y)$ and we can keep the same rate.
\end{remark}

\section{Proof Overview}
\label{sec:overview}
We explain the two different roles of the penalty and then show why the
single final projection is sufficient. The complete proofs, including all
parameter choices and probability bounds, are in the appendix.

\noindent\textbf{Convex Constraints: Control Infeasibility Before Projecting.}~The inner objective is strongly convex even when $f$ is nonsmooth.
A self-contained stochastic subgradient bound gives inner function error
$\widetilde O(1/K)$ (Lemma~\ref{lem:inner}). The softplus differs from
$\lambda_0[g]_+$ by at most $\gamma\log2$. Consequently, with a common
inner error $e=\widetilde O(1/K)$ and $y_t=Q_\ell(x_t)$,
\begin{equation}
    F_t(x_{t+1})+\lambda_0[g(x_{t+1})]_+
       \leq F_t(y_t)+e.
    \label{eq:overview-inner}
\end{equation}
Set $G_F:=G_f+4\ell$, a Lipschitz bound for $F_t$ on the unit ball.
The error bound for $g$ and the choice $\lambda_0\sigma=2G_F$ imply
\begin{equation}
    d_{t+1}:=\dist(x_{t+1},\C)\leq e/G_F,
    \qquad
    \norm{x_{t+1}-y_t}^2\leq2e/\ell.
    \label{eq:overview-feasibility}
\end{equation}
The first estimate controls infeasibility without computing a projection.
The second follows from strong convexity of the exact-penalty subproblem.

For the descent step, use $p_t=\Pset_{\C}(x_t)$ only as an analytical
comparison point. Since $y_t$ minimizes $F_t$ over $\C$,
\begin{equation}
    \ell\norm{x_{t+1}-x_t}^2
    \leq f(x_t)-f(x_{t+1})+G_f d_t+\ell d_t^2+e.
    \label{eq:overview-descent}
\end{equation}
 Summing
\eqref{eq:overview-descent} and using \eqref{eq:overview-feasibility}
controls the average squared outer displacement. For smooth $f$,
nonexpansiveness of the convex projection and Lipschitz continuity of
$\nabla f$ transfer this control to $\mathcal G_\ell(p_{t+1})^2$.
For nonsmooth $f$, Lipschitz continuity of $Q_\ell$ instead controls
$\mathcal S_\ell(p_{t+1})^2$. Both bounds have the form
$O(T^{-1}+e+e^2)$ for fixed problem parameters. Taking
$T=O(\epsilon^{-2})$ and $e=O(\epsilon^2)$ proves
Theorems~\ref{thm:convex-smooth} and~\ref{thm:convex-nonsmooth} after
random selection and one final projection.

\noindent\textbf{Nonconvex Constraints: Make the Proximal Witness Feasible.}~We do not try to make every penalized subproblem an exact reformulation of
a constrained proximal problem. Instead, $\Psi_\lambda$ is a fixed weakly
convex objective, and $\beta$ makes its proximal subproblems strongly
convex. Inexact proximal descent gives
\begin{equation}
    \frac1T\sum_{t=0}^{T-1}\mathcal R_{\lambda,\beta}(x_t)^2
    \leq \frac{4\beta\Delta_\lambda}{T}
         +4\beta\left(1+\frac\beta m\right)e.
    \label{eq:overview-nc-descent}
\end{equation}
This is the stochastic progress argument. The feasibility argument is
separate. If a proximal witness $y=Q_{\lambda,\beta}(x)$ were infeasible,
its optimality condition would read
\[
    \beta(x-y)=v+\lambda\nabla g(y)+n,
    \qquad v\in\partial f(y),\quad n\in N_{\B}(y).
\]
Assumption~\ref{ass:nonconvex} would then imply
$\mathcal R_{\lambda,\beta}(x)\geq
\lambda\sigma_{\rm out}-G_f=h_{\rm feas}$.
Hence a smaller residual forces $y\in\C$. Because $\C\subset\operatorname{int}\B$,
the artificial ball normal disappears, and the same optimality condition
is an approximate KKT condition for \eqref{eq:problem}.

Finally, for any nearest point $p\in\Pset_{\C}(x)$, $
    \norm{p-y}\leq\norm{p-x}+\norm{x-y}
       \leq2\norm{x-y}
       =\frac{2\mathcal R_{\lambda,\beta}(x)}{\beta}.$
Thus the final correction preserves proximity to the KKT witness without
requiring a globally nonexpansive nonconvex projection. The small residual
also places $x$ in the unique-projection tube. Combining this observation
with \eqref{eq:overview-nc-descent} proves Theorem~\ref{thm:nonconvex}.

\section{Experiments}
\label{sec:experiments}
We compare One Projection with Projected SGD on transfer learning and
classification with abstention, reporting means and sample standard
deviations over ten runs. Appendix~\ref{app:experimental-details} provides
the experimental settings, projection procedures, and additional results.

\subsection{Transfer Learning}
To adapt to a target distribution while retaining source performance, we
consider constrained empirical risk minimization
\citep{hanneke2024adaptive,NEURIPS2025_85d4d3ba}:
\[
    \min_{\theta} R_T(\theta)
    \qquad\text{subject to}\qquad R_S(\theta)\leq\tau_S,
\]
where $R_S,R_T$ are empirical logistic risks and
$\tau_S=1/\sqrt{n_S}$ for $n_S$ source observations.
We use a linear model on synthetic Gaussian data and a Softplus neural
network on ClimSim rainfall data \citep{yu2023climsim}, giving convex and
nonconvex reference problems, respectively.

Table~\ref{tab:transfer-results} shows comparable target errors and logistic
risks, with $6.9\times$ and $43.2\times$ lower mean runtime for One Projection.
Runtime includes source-reference fitting and constrained target optimization.
Final source risks satisfy the numerical tolerance; update budgets differ
across methods (Appendix~\ref{app:transfer-details}).
Figure~\ref{fig:transfer-training-runtime} summarizes runtimes for both tasks.
\begin{table}[!htbp]
\centering
\caption{Transfer learning: target test performance and training runtime
(mean $\pm$ sample standard deviation over ten runs). Update budgets and
source-risk diagnostics are given in Appendix~\ref{app:transfer-details}.}
\label{tab:transfer-results}
\small
\setlength{\tabcolsep}{5pt}
\renewcommand{\arraystretch}{1.12}
\begin{tabular}{@{}llccc@{}}
\toprule
Dataset & Method & \makecell{Target error\\(\%)} &
\makecell{Target logistic\\risk} & Runtime (s) \\
\midrule
\multirow{2}{*}{Gaussian} & One Projection & $10.00\pm0.81$ & $0.253\pm0.030$ & $9.74\pm0.23$ \\
 & Projected SGD & $9.92\pm0.66$ & $0.252\pm0.033$ & $67.55\pm17.90$ \\
\addlinespace[3pt]
\multirow{2}{*}{ClimSim} & One Projection & $7.26\pm1.31$ & $0.333\pm0.079$ & $15.25\pm2.12$ \\
 & Projected SGD & $7.13\pm1.86$ & $0.323\pm0.073$ & $658.91\pm116.26$ \\
\bottomrule
\end{tabular}
\end{table}

\definecolor{opblue}{HTML}{0072B2}
\definecolor{psgdorange}{HTML}{D55E00}
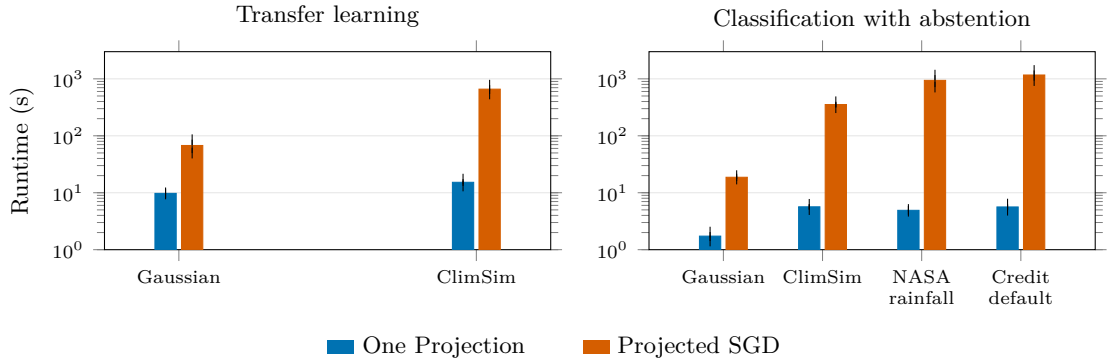
\begin{figure} [b]
\centering
\begin{tikzpicture}
\begin{groupplot}[
 group style={group size=2 by 1,horizontal sep=1.3cm},
 width=0.47\textwidth,height=4.2cm,
 ybar=2pt,/pgf/bar width=8pt,ymode=log,log origin=infty,
 ymin=1,ymax=3000,ytick={1,10,100,1000},
 ymajorgrids=true,grid style={gray!20},tick align=outside,
 tick label style={font=\scriptsize},title style={font=\small},
 label style={font=\small},xticklabel style={align=center},
 enlarge x limits=0.25,
 legend style={at={(0.5,1.04)},anchor=south,legend columns=2,
   font=\scriptsize,draw=none},
 error bars/y dir=both,error bars/y explicit,
 error bars/error bar style={draw=black},
 error bars/error mark options={draw=black,mark size=2pt}
]
\nextgroupplot[title={Transfer learning},ylabel={Runtime (s)},
 xtick={0,1},xticklabels={Gaussian,ClimSim}]
\addplot+[fill=opblue,draw=opblue] coordinates
 {(0,9.74) +- (0,0.23) (1,15.25) +- (0,2.12)};
\addplot+[fill=psgdorange,draw=psgdorange] coordinates
 {(0,67.55) +- (0,17.90) (1,658.91) +- (0,116.26)};
\nextgroupplot[title={Classification with abstention},
 xtick={0,1,2,3},xticklabels={Gaussian,ClimSim,{NASA\\rainfall},{Credit\\default}}]
\addplot+[fill=opblue,draw=opblue] coordinates
 {(0,1.73) +- (0,0.31) (1,5.66) +- (0,0.61)
  (2,4.91) +- (0,0.18) (3,5.63) +- (0,0.70)};
\addplot+[fill=psgdorange,draw=psgdorange] coordinates
 {(0,18.69) +- (0,1.30) (1,353.63) +- (0,42.62)
  (2,940.30) +- (0,223.84) (3,1167.85) +- (0,237.30)};
\end{groupplot}
\end{tikzpicture}

\smallskip
{\small\textcolor{opblue}{\rule{10pt}{5pt}}~One Projection
\qquad\textcolor{psgdorange}{\rule{10pt}{5pt}}~Projected SGD}
\caption{Training runtime over ten runs (mean $\pm$ sample standard deviation;
logarithmic scale). Transfer includes source-reference fitting. Update budgets
and optimizer-specific thresholds are given in Appendix~\ref{app:experimental-details}.}
\label{fig:transfer-training-runtime}
\label{fig:combined-training-runtime}
\end{figure}

Figure~\ref{fig:climsim-transfer-training-histories} shows the target validation
risk, target training risk, and source risk during optimization.
One Projection allows intermediate source-risk violations and restores
numerical feasibility at the terminal projection. Both methods reach
similar final target risks, with the update budgets reported in
Appendix~\ref{app:transfer-details}.
\input{ICLR/figures/transfer_climsim}

\subsection{Classification with Abstention}
Following \citet{kalan2026classification}, we learn scores $s_0,s_1$ that
accept class $j$ when $s_j(x)\geq0$ and abstain on simultaneous acceptance.
Using the hinge surrogate $\phi(u)=[1+u]_+$, we minimize ambiguity subject
to class-specific error constraints:
\[
    \min_{\theta_0,\theta_1}
        \widehat R_{\mathrm{amb},\phi}(\theta_0,\theta_1)
    \qquad\text{subject to}\qquad
        \widehat R_{\phi,j}(\theta_j)\leq\beta_j+\epsilon_j,
        \quad j\in\{0,1\}.
\]
Here $\widehat R_{\mathrm{amb},\phi}$ is the empirical product surrogate,
$\widehat R_{\phi,j}$ is the class-$j$ surrogate error, and $\epsilon_j$
is a constraint slack. The objective is generally nonconvex; neural-network
scores also yield nonconvex constraints.

We use linear scores on Gaussian data and neural networks on ClimSim,
NASA POWER rainfall \citep{stackhouse2021prediction}, and Give Me Some Credit
\citep{kaggle2011givemecredit}. At 500 updates, One Projection attains similar
errors with comparable or lower ambiguity, using $10.8$--$207.4\times$ less
runtime (Table~\ref{tab:combined-abstention-results}).
Figure~\ref{fig:climsim-rainfall-training-histories} shows the ClimSim training
curves. These compare validation-calibrated tradeoffs: some surrogate
thresholds differ by optimizer, and nominal test-error targets are not
uniformly met (Appendix~\ref{app:abstention-details}).
\begin{table}[!htbp]
\centering
\caption{Classification with abstention: test errors, ambiguity, and training
runtime (mean $\pm$ sample standard deviation over ten runs).
OP denotes One Projection and PSGD denotes Projected SGD.
The validation-calibrated thresholds are listed in
Table~\ref{tab:abstention-calibration}.}
\label{tab:combined-abstention-results}
\small
\setlength{\tabcolsep}{4pt}
\renewcommand{\arraystretch}{1.12}
\begin{tabular}{@{}llcccc@{}}
\toprule
Dataset & Method & \makecell{Class 0\\error (\%)} &
\makecell{Class 1\\error (\%)} & Ambiguity (\%) & Runtime (s) \\
\midrule
\multirow{2}{*}{Gaussian} & OP & $12.84\pm2.60$ & $13.89\pm2.59$ & $11.10\pm2.28$ & $1.73\pm0.31$ \\
 & PSGD & $12.99\pm2.50$ & $14.04\pm2.24$ & $11.15\pm2.53$ & $18.69\pm1.30$ \\
\addlinespace[3pt]
\multirow{2}{*}{ClimSim} & OP & $9.35\pm0.43$ & $5.00\pm1.00$ & $1.72\pm0.44$ & $5.66\pm0.61$ \\
 & PSGD & $9.00\pm0.86$ & $5.09\pm1.73$ & $2.95\pm1.03$ & $353.63\pm42.62$ \\
\addlinespace[3pt]
\multirow{2}{*}{NASA rainfall} & OP & $9.68\pm0.42$ & $9.55\pm1.74$ & $9.17\pm0.76$ & $4.91\pm0.18$ \\
 & PSGD & $9.80\pm0.41$ & $10.74\pm2.84$ & $8.62\pm1.75$ & $940.30\pm223.84$ \\
\addlinespace[3pt]
\multirow{2}{*}{Credit default} & OP & $16.75\pm2.27$ & $21.20\pm2.54$ & $14.95\pm2.89$ & $5.63\pm0.70$ \\
 & PSGD & $17.80\pm2.04$ & $20.44\pm2.96$ & $16.05\pm2.95$ & $1167.85\pm237.30$ \\
\bottomrule
\end{tabular}
\end{table}

\input{ICLR/figures/abstention_climsim}

\FloatBarrier
\section{Conclusion}
We presented a penalized proximal subgradient method for nonconvex
stochastic optimization with one terminal projection. The method handles
convex and nonconvex constraints with $\widetilde O(\epsilon^{-4})$
stochastic subgradient calls under the corresponding regularity assumptions.
Experiments on transfer learning and classification with abstention show
comparable predictive performance to Projected SGD at substantially lower
runtime.

\clearpage
\bibliographystyle{plainnat}
\bibliography{ICLR/one_projection}

@misc{kalan2026classification,
  title         = {Classification with Abstention Under Class-Conditional Error Constraints},
  author        = {Kalan, Mohammadreza M. and Deng, Yuyang and Hamidi, Sanaz},
  year          = {2026},
  eprint        = {2609.22632},
  url           = {https://arxiv.org/abs/2609.22632}
}

@inproceedings{mahdavi2012one,
  author = {Mehrdad Mahdavi and Tianbao Yang and Rong Jin and Shenghuo Zhu and Jinfeng Yi},
  title = {Stochastic Gradient Descent with Only One Projection},
  booktitle = {Advances in Neural Information Processing Systems},
  volume = {25},
  year = {2012},
  url = {https://papers.neurips.cc/paper_files/paper/2012/hash/c52f1bd66cc19d05628bd8bf27af3ad6-Abstract.html}
}

@inproceedings{yang2017richer,
  author = {Tianbao Yang and Qihang Lin and Lijun Zhang},
  title = {A Richer Theory of Convex Constrained Optimization with Reduced Projections and Improved Rates},
  booktitle = {Proceedings of the 34th International Conference on Machine Learning},
  series = {Proceedings of Machine Learning Research},
  volume = {70},
  pages = {3901--3910},
  year = {2017},
  publisher = {PMLR},
  url = {https://proceedings.mlr.press/v70/yang17f.html}
}

@article{hanneke2024adaptive,
  title={Adaptive sample aggregation in transfer learning},
  author={Hanneke, Steve and Kpotufe, Samory},
  journal={arXiv preprint arXiv:2408.16189},
  year={2024}
}

@inproceedings{jaggi2013frankwolfe,
  author = {Martin Jaggi},
  title = {Revisiting {Frank--Wolfe}: Projection-Free Sparse Convex Optimization},
  booktitle = {Proceedings of the 30th International Conference on Machine Learning},
  series = {Proceedings of Machine Learning Research},
  volume = {28},
  pages = {427--435},
  year = {2013},
  publisher = {PMLR},
  url = {https://proceedings.mlr.press/v28/jaggi13.html}
}

@inproceedings{NEURIPS2025_85d4d3ba,
 author = {Deng, Yuyang and Kpotufe, Samory},
 booktitle = {Advances in Neural Information Processing Systems},
 doi = {10.52202/085713-3108},
 editor = {D. Belgrave and C. Zhang and H. Lin and R. Pascanu and P. Koniusz and M. Ghassemi and N. Chen},
 pages = {92917--92954},
 publisher = {Curran Associates, Inc.},
 title = {Mixed-Sample SGD: an End-to-end Analysis of Supervised Transfer Learning},
 url = {https://proceedings.neurips.cc/paper_files/paper/2025/file/85d4d3ba2d10ca59a3cff87d5692474f-Paper-Conference.pdf},
 volume = {38, Main Conference},
 year = {2025}
}

@inproceedings{
kalan2026neymanpearson,
title={Neyman-Pearson Classification under Both Null and Alternative Distributions Shift},
author={Mohammadreza Mousavi Kalan and Yuyang Deng and Eitan J. Neugut and Samory Kpotufe},
booktitle={The Fourteenth International Conference on Learning Representations},
year={2026},
url={https://openreview.net/forum?id=pHckxhmBlI}
}

@article{davis2019guided,
  author = {Damek Davis and Benjamin Grimmer},
  title = {Proximally Guided Stochastic Subgradient Method for Nonsmooth, Nonconvex Problems},
  journal = {SIAM Journal on Optimization},
  volume = {29},
  number = {3},
  pages = {1908--1930},
  year = {2019},
  doi = {10.1137/17M1151031},
  url = {https://doi.org/10.1137/17M1151031}
}

@article{davis2019model,
  author = {Damek Davis and Dmitriy Drusvyatskiy},
  title = {Stochastic Model-Based Minimization of Weakly Convex Functions},
  journal = {SIAM Journal on Optimization},
  volume = {29},
  number = {1},
  pages = {207--239},
  year = {2019},
  doi = {10.1137/18M1178244},
  url = {https://doi.org/10.1137/18M1178244}
}

@article{davis2025sets,
  author = {Damek Davis and Dmitriy Drusvyatskiy and Zhan Shi},
  title = {Stochastic Optimization Over Proximally Smooth Sets},
  journal = {SIAM Journal on Optimization},
  volume = {35},
  number = {1},
  pages = {157--179},
  year = {2025},
  doi = {10.1137/20M1320225},
  url = {https://doi.org/10.1137/20M1320225}
}

@inproceedings{ma2020quadratic,
  author = {Runchao Ma and Qihang Lin and Tianbao Yang},
  title = {Quadratically Regularized Subgradient Methods for Weakly Convex Optimization with Weakly Convex Constraints},
  booktitle = {Proceedings of the 37th International Conference on Machine Learning},
  series = {Proceedings of Machine Learning Research},
  volume = {119},
  pages = {6554--6564},
  year = {2020},
  publisher = {PMLR},
  url = {https://proceedings.mlr.press/v119/ma20d.html}
}

@misc{boob2019functional,
  author = {Digvijay Boob and Qi Deng and Guanghui Lan},
  title = {Stochastic First-order Methods for Convex and Nonconvex Functional Constrained Optimization},
  year = {2019},
  howpublished = {arXiv:1908.02734},
  url = {https://arxiv.org/abs/1908.02734}
}

@misc{jia2022firstorder,
  author = {Zhichao Jia and Benjamin Grimmer},
  title = {First-Order Methods for Nonsmooth Nonconvex Functional Constrained Optimization with or without {Slater} Points},
  year = {2022},
  howpublished = {arXiv:2212.00927},
  url = {https://arxiv.org/abs/2212.00927}
}

@misc{liu2025spider,
  author = {Wei Liu and Yangyang Xu},
  title = {A {SPIDER}-type Stochastic Subgradient Method for Expectation-constrained Nonconvex Nonsmooth Optimization},
  year = {2025},
  howpublished = {arXiv:2501.19214v2},
  url = {https://arxiv.org/abs/2501.19214v2}
}

@inproceedings{zhang2025alignment,
  title     = {Alignment of Large Language Models with Constrained Learning},
  author    = {Zhang, Botong and Li, Shuo and Hounie, Ignacio and
               Bastani, Osbert and Ding, Dongsheng and Ribeiro, Alejandro},
  booktitle = {Advances in Neural Information Processing Systems},
  volume    = {38},
  year      = {2025},
  doi       = {10.52202/085713-1042}
}

@inproceedings{peng2025enhancing,
  title     = {Enhancing Safety in Reinforcement Learning with Human Feedback
               via Rectified Policy Optimization},
  author    = {Peng, Xiyue and Guo, Hengquan and Zhang, Jiawei and
               Zou, Dongqing and Shao, Ziyu and Wei, Honghao and Liu, Xin},
  booktitle = {Advances in Neural Information Processing Systems},
  volume    = {38},
  year      = {2025},
  doi       = {10.52202/085713-2120}
}

@inproceedings{hu2022lora,
  title     = {{LoRA}: Low-Rank Adaptation of Large Language Models},
  author    = {Hu, Edward J. and Shen, Yelong and Wallis, Phillip and
               Allen-Zhu, Zeyuan and Li, Yuanzhi and Wang, Shean and
               Wang, Lu and Chen, Weizhu},
  booktitle = {International Conference on Learning Representations},
  year      = {2022},
  url       = {https://openreview.net/forum?id=nZeVKeeFYf9}
}

@inproceedings{zhang2023adalora,
  title     = {Adaptive Budget Allocation for Parameter-Efficient Fine-Tuning},
  author    = {Zhang, Qingru and Chen, Minshuo and Bukharin, Alexander and
               He, Pengcheng and Cheng, Yu and Chen, Weizhu and Zhao, Tuo},
  booktitle = {International Conference on Learning Representations},
  year      = {2023},
  url       = {https://openreview.net/forum?id=lq62uWRJjiY}
}

@inproceedings{jang2024lora,
  title     = {{LoRA} Training in the {NTK} Regime has No Spurious Local Minima},
  author    = {Jang, Uijeong and Lee, Jason D. and Ryu, Ernest K.},
  booktitle = {Proceedings of the 41st International Conference on Machine Learning},
  series    = {Proceedings of Machine Learning Research},
  volume    = {235},
  pages     = {21306--21328},
  year      = {2024},
  publisher = {PMLR},
  url       = {https://proceedings.mlr.press/v235/jang24d.html}
}

@inproceedings{park2025riemannian,
  title     = {Riemannian Optimization for {LoRA} on the Stiefel Manifold},
  author    = {Park, JuneYoung and Kang, Minjae and Lee, Seongbae and
               Lee, Haegang and Kim, Seongwan and Lee, Jaeho},
  booktitle = {Findings of the Association for Computational Linguistics:
               EMNLP 2025},
  pages     = {20971--20985},
  year      = {2025},
  address   = {Suzhou, China},
  publisher = {Association for Computational Linguistics},
  doi       = {10.18653/v1/2025.findings-emnlp.1143},
  url       = {https://aclanthology.org/2025.findings-emnlp.1143/}
}

@article{arjevani2023lower,
  author = {Yossi Arjevani and Yair Carmon and John C. Duchi and Dylan J. Foster and Nathan Srebro and Blake Woodworth},
  title = {Lower Bounds for Non-Convex Stochastic Optimization},
  journal = {Mathematical Programming},
  volume = {199},
  number = {1--2},
  pages = {165--214},
  year = {2023},
  doi = {10.1007/s10107-022-01822-7},
  url = {https://doi.org/10.1007/s10107-022-01822-7}
}

@misc{jin2026bounded,
  author = {Jikai Jin},
  title = {A Tight Lower Bound for Smooth Nonconvex Stochastic Optimization with Bounded Gradient Noise},
  year = {2026},
  howpublished = {arXiv:2608.09004},
  url = {https://arxiv.org/abs/2608.09004}
}

@inproceedings{yu2023climsim,
  title     = {ClimSim: A Large Multi-Scale Dataset for Hybrid Physics--ML
               Climate Emulation},
  author    = {Yu, Sungduk and Hannah, Walter and Peng, Liran and Lin, Jerry
               and Bhouri, Mohamed Aziz and Gupta, Ritwik and L{\"u}tjens,
               Bj{\"o}rn and Will, Justus C. and Behrens, Gunnar and Busecke,
               Julius J. M. and others},
  booktitle = {Advances in Neural Information Processing Systems},
  volume    = {36},
  year      = {2023},
  doi       = {10.52202/075280-0968},
  url       = {https://papers.neurips.cc/paper_files/paper/2023/hash/45fbcc01349292f5e059a0b8b02c8c3f-Abstract-Datasets_and_Benchmarks.html}
}

@inproceedings{stackhouse2021prediction,
  title={The prediction of worldwide energy resources (POWER) project},
  author={Stackhouse Jr, Paul W and Macpherson, Bradley and Broddle, Madison and McNeil, Chequel and Barnett, Jason and Mikovitz, Colleen and Zhang, Taiping},
  booktitle={United Nations Climate Change Conference, Community of Practice 26},
  year={2021}
}

@misc{kaggle2011givemecredit,
  author       = {{Kaggle}},
  title        = {{Give Me Some Credit}},
  year         = {2011},
  howpublished = {Kaggle competition dataset},
  url          = {https://www.kaggle.com/competitions/GiveMeSomeCredit},
  note         = {Accessed September 25, 2026}
}

\appendix
\numberwithin{table}{section}
\section{A High-Probability Inner-Solver Bound}
\label{app:inner}

We first give a bound that applies to both penalty choices. All probability
statements can be conditioned on the history before an inner run. This is
important because the next proximal objective depends on the preceding
stochastic iterates.

\begin{lemma}[Strongly convex stochastic subgradient subroutine]
\label{lem:inner}
Let $H$ be $m$-strongly convex on a neighborhood of $\B$, where $m>0$.
Suppose $\widetilde h_s$ is a stochastic subgradient at $z_s\in\B$ with
$\E[\widetilde h_s\mid\mathcal F_s]\in\partial H(z_s)$ and
$\norm{\widetilde h_s}\leq G$ almost surely, for some $G>0$.
Starting at any $z_1\in\B$, set
\[
    z_{s+1}=\Pset_{\B}(z_s-\eta_s\widetilde h_s),
    \qquad \eta_s=\frac{2}{m(s+1)},\qquad
    \bar z=\frac1K\sum_{s=1}^Kz_s.
\]
For $0<\delta<1$, with probability at least $1-\delta$,
\begin{equation}
    H(\bar z)-\min_{z\in\B}H(z)
    \leq U_K(m,G,\delta)
    :=\frac{m+(G^2/m)\{\log(K+1)+2\log(1/\delta)\}}{K}.
    \label{eq:inner-bound}
\end{equation}
\end{lemma}

\begin{proof}
Let $u$ minimize $H$ over $\B$, put $D_s=\norm{z_s-u}$, and write
$h_s=\E[\widetilde h_s\mid\mathcal F_s]$ and
$X_s=\ip{h_s-\widetilde h_s}{z_s-u}$.
Nonexpansiveness of the ball projection and strong convexity imply
\begin{align*}
    H(z_s)-H(u)
    &\leq \ip{h_s}{z_s-u}-\frac m2D_s^2\\
    &\leq\frac{D_s^2-D_{s+1}^2}{2\eta_s}
            -\frac m2D_s^2+\frac{\eta_sG^2}{2}+X_s.
\end{align*}
Summing with the stated stepsize and dropping the negative terminal term gives
\begin{equation}
    \sum_{s=1}^K(H(z_s)-H(u))
    \leq\frac m4D_1^2-\frac m4\sum_{s=1}^KD_s^2
       +\frac{G^2}{m}\sum_{s=1}^K\frac1{s+1}+\sum_{s=1}^KX_s.
    \label{eq:inner-telescope}
\end{equation}
Conditionally on $\mathcal F_s$, $X_s$ is centered and its range has length
at most $2GD_s$. The conditional bounded-variable exponential inequality gives
\[
    \E[\exp(\theta X_s)\mid\mathcal F_s]
       \leq\exp(\theta^2G^2D_s^2/2).
\]
Iterating conditional expectations and applying Markov's inequality shows
that, with probability at least $1-\delta$,
\[
    \sum_{s=1}^KX_s
       \leq \frac{\theta G^2}{2}\sum_{s=1}^KD_s^2
                   +\frac{\log(1/\delta)}{\theta}.
\]
Set $\theta=m/(2G^2)$. The squared-distance sum cancels the negative sum
in \eqref{eq:inner-telescope}. Since $D_1\leq2$ and
$\sum_{s=1}^K(s+1)^{-1}\leq\log(K+1)$, division by $K$ and convexity of
$H$ prove \eqref{eq:inner-bound}.
\end{proof}

There is no differentiability requirement on $H$. In particular, neither
the small softplus smoothing parameter nor the kink of the hinge changes
this subgradient bound. The function error is
$\widetilde O((m+G^2/m)/K)$, rather than a smooth-gradient bound involving
the penalty's smoothness constant.

\section{Proofs for Convex Constraints}
\label{app:convex}

Throughout this section Assumption~\ref{ass:convex} holds. Define
\begin{equation}
    G_F:=G_f+4\ell,\qquad
    G_H:=G_F+\lambda_0G_g,\qquad
    e_K(\delta):=U_K(\ell,G_H,\delta)+\frac{\log2}{K}.
    \label{eq:convex-error-constants}
\end{equation}
Thus $\lambda_0\sigma-G_F=G_F>0$ under
\eqref{eq:convex-settings}. This positive margin is an explicit consequence
of the parameter choice, not an additional assumption.

\subsection{Constraint Error Bound and Exact Penalization}

\begin{lemma}[Error bound induced by the convex defining function]
\label{lem:convex-eb}
For every $x\in\B$,
\begin{equation}
    \dist(x,\C)\leq\frac{[g(x)]_+}{\sigma}.
    \label{eq:convex-error-bound}
\end{equation}
\end{lemma}

\begin{proof}
The claim is immediate for $x\in\C$. Otherwise let $p=\Pset_{\C}(x)$.
Compactness of $\C$ gives existence, and convexity gives uniqueness.
The point $p$ lies on $g=0$. Because $\nabla g(p)\ne0$, the normal cone
is $N_{\C}(p)=\{\alpha\nabla g(p):\alpha\geq0\}$.
The projection condition gives $x-p=\alpha\nabla g(p)$ for some
$\alpha>0$. Convexity then yields
\[
    g(x)\geq g(p)+\ip{\nabla g(p)}{x-p}
       =\norm{\nabla g(p)}\norm{x-p}
       \geq\sigma\dist(x,\C).
\]
\end{proof}

\begin{lemma}[Inner accuracy, infeasibility, and proximal accuracy]
\label{lem:convex-inner}
Fix $x_t\in\B$, let $y_t=Q_\ell(x_t)$, and run the inner routine with the
convex settings. With conditional probability at least $1-\delta$,
\begin{align}
    F_t(x_{t+1})+\lambda_0[g(x_{t+1})]_+
       &\leq F_t(y_t)+e_K(\delta),\label{eq:cv-inner-gap}\\
    \dist(x_{t+1},\C)&\leq e_K(\delta)/G_F,\label{eq:cv-inner-dist}\\
    \norm{x_{t+1}-y_t}^2&\leq2e_K(\delta)/\ell.\label{eq:cv-prox-error}
\end{align}
\end{lemma}

\begin{proof}
The function $F_t$ is $\ell$-strongly convex. Since $p_\gamma$ is convex
and nondecreasing and $g$ is convex,
$H_t:=F_t+p_\gamma\circ g$ is also $\ell$-strongly convex.
The stochastic subgradient used in Algorithm~\ref{alg:one-projection}
has norm at most
\[
    G_f+2\ell\norm{z_s-x_t}+\lambda_0G_g
       \leq G_f+4\ell+\lambda_0G_g=G_H.
\]
Apply Lemma~\ref{lem:inner} and compare the minimum of $H_t$ to $H_t(y_t)$.
The elementary bounds
\[
    \lambda_0[u]_+\leq p_\gamma(u)
       \leq\lambda_0[u]_++\gamma\log2
\]
and $\gamma=1/K$ prove \eqref{eq:cv-inner-gap}.

For completeness, the exact-penalty objective
$P_t(z):=F_t(z)+\lambda_0[g(z)]_+$ has the same unique minimizer $y_t$
over $\B$ as $F_t$ does over $\C$. Indeed, for any $z\in\B$, set
$p=\Pset_{\C}(z)$ and $d=\norm{z-p}$. The function $F_t$ is $G_F$-Lipschitz
on $\B$, so Lemma~\ref{lem:convex-eb} gives
\begin{align}
    P_t(z)&\geq F_t(p)-G_Fd+\lambda_0\sigma d\nonumber\\
          &\geq F_t(y_t)+G_Fd.\label{eq:cv-exact-penalty}
\end{align}
This proves exactness and, together with \eqref{eq:cv-inner-gap},
\eqref{eq:cv-inner-dist}. Finally, $P_t$ is $\ell$-strongly convex.
At its minimizer over the convex set $\B$, quadratic growth gives
\[
    \frac\ell2\norm{x_{t+1}-y_t}^2
       \leq P_t(x_{t+1})-P_t(y_t)\leq e_K(\delta),
\]
proving \eqref{eq:cv-prox-error}.
\end{proof}

\subsection{Descent and Stationarity Transfer}

All projections $p_t=\Pset_{\C}(x_t)$ and proximal points below are proof
objects; only the selected terminal projection is computed.
Write $d_t=\norm{x_t-p_t}$ and $s_t=\norm{x_{t+1}-x_t}$.

\begin{lemma}[Outer descent with infeasible iterates]
\label{lem:cv-descent}
Suppose the conclusions of Lemma~\ref{lem:convex-inner} hold in every
outer iteration with a common error bound $e$. Then
\begin{equation}
    \frac1T\sum_{t=0}^{T-1}s_t^2
    \leq\frac{\Delta_f}{\ell T}
       +\frac{1+G_f/G_F}{\ell}e+\frac{e^2}{G_F^2}.
    \label{eq:cv-average-steps}
\end{equation}
\end{lemma}

\begin{proof}
Dropping the nonnegative penalty in \eqref{eq:cv-inner-gap}, using
$F_t(y_t)\leq F_t(p_t)$, and using Lipschitz continuity of $f$ give
\begin{align*}
    f(x_{t+1})+\ell s_t^2
    &\leq F_t(y_t)+e\\
    &\leq f(p_t)+\ell d_t^2+e\\
    &\leq f(x_t)+G_f d_t+\ell d_t^2+e.
\end{align*}
This proves \eqref{eq:overview-descent}, including its quadratic distance
term. Now sum over $t$, use $d_0=0$ and $d_t\leq e/G_F$ for $t\geq1$,
and use $f(x_T)\geq\min_{\B}f$.
\end{proof}

\begin{lemma}[Properties of the convex constrained proximal map]
\label{lem:cv-prox-lipschitz}
The point $Q_\ell(x)$ is unique and $Q_\ell$ is $2$-Lipschitz. Moreover,
$\Phi_\ell$ is differentiable with
$\nabla\Phi_\ell(x)=2\ell(x-Q_\ell(x))$, and
\eqref{eq:moreau-intuition} holds.
\end{lemma}

\begin{proof}
The function $f+\ind_{\C}$ is proper, lower semicontinuous and
$\ell$-weakly convex. Compactness gives existence of the proximal minimizer;
strong convexity gives uniqueness. For $q_i=Q_\ell(x_i)$, the optimality
conditions and monotonicity of the subdifferential of
$f+\ind_{\C}+\ell\norm{\cdot}^2/2$ give
\[
    \ip{2\ell(x_1-q_1)-2\ell(x_2-q_2)}{q_1-q_2}
        \geq-\ell\norm{q_1-q_2}^2.
\]
Consequently,
$\norm{q_1-q_2}^2\leq2\ip{x_1-x_2}{q_1-q_2}$, which proves the
Lipschitz bound. Differentiating the minimum with its unique continuous
minimizer yields the stated envelope gradient.
The proximal optimality condition is
\[
    2\ell(x-Q_\ell(x))\in
         \partial f(Q_\ell(x))+N_{\C}(Q_\ell(x)).
\]
The sum rule is valid since $f$ is locally Lipschitz and weakly convex on
a neighborhood of $\C$; equivalently it is the convex sum rule after
adding $\ell\norm{\cdot}^2/2$. This proves
\eqref{eq:moreau-intuition}.
\end{proof}

\begin{lemma}[The final projection preserves a stationarity certificate]
\label{lem:cv-transfer}
Under the event in Lemma~\ref{lem:cv-descent},
\begin{align}
    \frac1T\sum_{t=0}^{T-1}\mathcal S_\ell(p_{t+1})^2
    &\leq\frac{48\ell\Delta_f}{T}
      +(72+48G_f/G_F)\ell e
      +\frac{156\ell^2e^2}{G_F^2},\label{eq:cv-nonsmooth-transfer}\\
    \frac1T\sum_{t=0}^{T-1}\mathcal G_\ell(p_{t+1})^2
    &\leq\frac{12\ell\Delta_f}{T}
      +(162+12G_f/G_F)\ell e
      +\frac{39\ell^2e^2}{G_F^2},\label{eq:cv-smooth-transfer}
\end{align}
where the second statement additionally assumes $f$ is $\ell$-smooth.
\end{lemma}

\begin{proof}
Put $E_t=\norm{x_{t+1}-y_t}$, so $E_t^2\leq2e/\ell$.
For the nonsmooth measure, Lemma~\ref{lem:cv-prox-lipschitz} gives
\begin{align*}
    \mathcal S_\ell(p_{t+1})
    &\leq2\ell\norm{p_{t+1}-y_t}
           +4\ell\norm{x_t-p_{t+1}}\\
    &\leq4\ell s_t+2\ell E_t+6\ell d_{t+1}.
\end{align*}
Squaring with $(a+b+c)^2\leq3(a^2+b^2+c^2)$, averaging, and applying
\eqref{eq:cv-average-steps} proves \eqref{eq:cv-nonsmooth-transfer}.

For the smooth measure, nonexpansiveness of the convex projection and
$\ell$-Lipschitz continuity of $\nabla f$ show that
$\mathcal G_\ell$ is $3\ell$-Lipschitz on $\B$. Proximal optimality also gives
\[
    y_t=\Pset_{\C}\left(
       y_t-\frac{\nabla f(y_t)+2\ell(y_t-x_t)}{\ell}\right),
    \qquad
    \mathcal G_\ell(y_t)\leq2\ell\norm{y_t-x_t}.
\]
It follows that
\[
    \mathcal G_\ell(p_{t+1})
      \leq2\ell s_t+5\ell E_t+3\ell d_{t+1}.
\]
The same squared-sum bound proves \eqref{eq:cv-smooth-transfer}.
\end{proof}

\subsection{Shared Schedule and Proofs of the Convex-Constraint Theorems}
\label{app:convex-schedule}

Theorems~\ref{thm:convex-smooth} and~\ref{thm:convex-nonsmooth} use the
same schedule. Choose
\begin{equation}
\begin{split}
    T&:=\max\left\{2,\left\lceil
                       \frac{1000\ell\Delta_f}{\epsilon^2}
                          \right\rceil\right\},\qquad
    \delta_{\rm in}:=\frac1{6T},\\
    e_*&:=\min\left\{\frac{\epsilon^2}{4000\ell},
                         \frac{G_F\epsilon}{60\ell}\right\}.
\end{split}
\label{eq:cv-schedule}
\end{equation}
Take any $K\geq2$ such that $e_K(\delta_{\rm in})\leq e_*$. This is an
explicit scalar inequality involving only the specified constants.
For example, double $K$ from $2$ until it holds. There is such a choice with
\begin{equation}
    K=\widetilde O\left(
       1+\left(1+\ell+\frac{G_H^2}{\ell}\right)
         \max\left\{\frac{\ell}{\epsilon^2},
                          \frac{\ell}{G_F\epsilon}\right\}\right).
    \label{eq:cv-K-explicit}
\end{equation}
The logarithmic factors include $\log T$ from $\delta_{\rm in}$.

Condition on the history at the start of each inner run and apply
Lemma~\ref{lem:convex-inner}. Let $\mathcal E_{\rm in}$ be the event that
all $T$ inner guarantees hold with error at most $e_*$. A union bound gives
$\Pr(\mathcal E_{\rm in})\geq1-T\delta_{\rm in}=5/6$.

\begin{proof}[Proof of Theorem~\ref{thm:convex-smooth}]
On $\mathcal E_{\rm in}$, \eqref{eq:cv-smooth-transfer},
\eqref{eq:cv-schedule}, and $G_f/G_F\leq1$ give
\[
    \frac1T\sum_{t=0}^{T-1}\mathcal G_\ell(p_{t+1})^2
       \leq\left(\frac{12}{1000}
               +\frac{174}{4000}
               +\frac{39}{3600}\right)\epsilon^2
       <\frac{\epsilon^2}{6}.
\]
Conditionally on any such trajectory, uniform independent selection of
$J\in\{1,\ldots,T\}$ and Markov's inequality imply
$\Pr\{\mathcal G_\ell(p_J)\leq\epsilon\mid\text{trajectory}\}\geq5/6$.
The unconditional success probability is at least $25/36>2/3$.
The actual output $\hat x=p_J$ is always feasible and requires just one
projection onto $\C$. The schedule gives
$TK=\widetilde O(\epsilon^{-4})$ stochastic gradient calls for fixed
problem parameters.
\end{proof}

\begin{proof}[Proof of Theorem~\ref{thm:convex-nonsmooth}]
On the same event $\mathcal E_{\rm in}$,
\eqref{eq:cv-nonsmooth-transfer} and \eqref{eq:cv-schedule} give
\[
    \frac1T\sum_{t=0}^{T-1}\mathcal S_\ell(p_{t+1})^2
       \leq\left(\frac{48}{1000}
               +\frac{120}{4000}
               +\frac{156}{3600}\right)\epsilon^2
       <\frac{\epsilon^2}{6}.
\]
Uniform independent selection of $J$ and Markov's inequality again give
$\Pr\{\mathcal S_\ell(p_J)\leq\epsilon\}\geq25/36>2/3$.
The output is exactly feasible, and \eqref{eq:moreau-intuition} supplies
the nearby-stationarity interpretation. The same schedule and final
projection yield the stated stochastic subgradient and projection counts.
\end{proof}

The gap $\Delta_f$ need not be computed: any known upper bound may replace
it in the schedule, and $\Delta_f\leq2G_f$ on the unit ball.
No evaluation of the stationarity measure or of the proof-only points
$p_t,y_t$ is needed to run the algorithm.

\section{Proofs for Nonconvex Constraints}
\label{app:nonconvex}

Throughout this section Assumption~\ref{ass:nonconvex} holds.
Set
\[
    \rho_\lambda:=\ell+\lambda L_g,\qquad
    G_\lambda:=G_f+\lambda G_g,\qquad
    G_H:=G_\lambda+2\beta,\qquad m=\beta-\rho_\lambda>0.
\]
The constants $G_H,m$ in this section refer to the nonconvex configuration.

\subsection{Geometry of the Terminal Projection}

\begin{lemma}[A projection tube derived from $g$]
\label{lem:geometry}
Let $R=\sigma_{\rm bd}/L_g$. For every $p,y\in\C$ and
$n\in N_{\C}(p)$,
\begin{equation}
    \ip{n}{y-p}\leq\frac{\norm{n}}{2R}\norm{y-p}^2.
    \label{eq:normal-inequality}
\end{equation}
Consequently, $\Pset_{\C}(x)$ is unique whenever $\dist(x,\C)<R$.
In particular, $\C$ is $R$-proximally smooth.
\end{lemma}

\begin{proof}
At an interior point of $\C$ the normal cone is $\{0\}$.
At a boundary point $p$, the nonzero gradient condition yields the smooth
inequality normal formula
\[
    N_{\C}(p)=\{\alpha\nabla g(p):\alpha\geq0\}.
\]
Indeed, the local tangent cone is the halfspace
$\{u:\ip{\nabla g(p)}{u}\leq0\}$, whose polar is this normal ray.
Since $g(p)=0$ and $g(y)\leq0$, the smooth lower bound along the segment
in $\B$ gives
\[
    g(y)\geq\ip{\nabla g(p)}{y-p}
                   -\frac{L_g}{2}\norm{y-p}^2,
    \qquad
    \ip{\nabla g(p)}{y-p}\leq\frac{L_g}{2}\norm{y-p}^2.
\]
Multiply by $\alpha$ and use
$\alpha\leq\norm{n}/\sigma_{\rm bd}$ to obtain
\eqref{eq:normal-inequality}.

To see uniqueness directly, suppose $p_1,p_2$ are projections of $x$,
with common distance $d<R$. Then $n_i=x-p_i\in N_{\C}(p_i)$.
Applying \eqref{eq:normal-inequality} twice and adding yields
\[
    \norm{p_1-p_2}^2
       \leq\frac dR\norm{p_1-p_2}^2.
\]
Thus $p_1=p_2$. Projections exist by compactness, so this proves the
unique-projection tube characterization of proximal smoothness.
\end{proof}

\subsection{Penalty Residual and Original-Problem Stationarity}

The function $[g]_+$ is $L_g$-weakly convex because
\[
    [g(x)]_++\frac{L_g}{2}\norm{x}^2
       =\max\left\{g(x)+\frac{L_g}{2}\norm{x}^2,
                         \frac{L_g}{2}\norm{x}^2\right\}
\]
is convex. Thus $\Psi_\lambda$ is $\rho_\lambda$-weakly convex, and
$Q_{\lambda,\beta}(x)$ is well-defined and unique. Since $f$ is Lipschitz
on a neighborhood of $\B$, all of its limiting subgradients on $\B$ have
norm at most $G_f$.

\begin{lemma}[A small penalty residual has a feasible KKT witness]
\label{lem:nc-transfer}
For $x\in\B$, write $y=Q_{\lambda,\beta}(x)$ and
$r=\mathcal R_{\lambda,\beta}(x)$. If $r<h_{\rm feas}$, then $y\in\C$ and
there are $v\in\partial f(y)$ and $\mu\in[0,\lambda]$ with
\begin{equation}
    \beta(x-y)=v+\mu\nabla g(y),\qquad \mu g(y)=0.
    \label{eq:nc-witness}
\end{equation}
Every $p\in\Pset_{\C}(x)$ then satisfies
\begin{equation}
    \dist(x,\C)\leq\frac r\beta,\qquad
    \norm{p-y}\leq\frac{2r}{\beta},\qquad
    \dist(0,\partial f(y)+N_{\C}(y))\leq r.
    \label{eq:nc-transfer}
\end{equation}
If additionally $r<\beta R$, the projection is unique.
\end{lemma}

\begin{proof}
Convex subdifferential calculus after a quadratic shift, or the equivalent
weakly convex calculus, gives
\begin{equation}
    \beta(x-y)=v+\lambda\theta\nabla g(y)+n,
    \qquad v\in\partial f(y),\quad n\in N_{\B}(y),
    \label{eq:nc-prox-opt}
\end{equation}
where $\theta=1$ for $g(y)>0$, $\theta=0$ for $g(y)<0$, and
$\theta\in[0,1]$ for $g(y)=0$.
If $g(y)>0$, then $N_{\B}(y)$ is a cone and
\[
    r=\norm{v+\lambda\nabla g(y)+n}
       \geq\lambda\dist(0,\nabla g(y)+N_{\B}(y))-G_f
       \geq h_{\rm feas},
\]
a contradiction. Thus $y\in\C\subset\operatorname{int}\B$ and $n=0$.
Set $\mu=\lambda\theta$. This proves \eqref{eq:nc-witness}, including
nonnegativity and exact complementarity. The boundary normal formula
from Lemma~\ref{lem:geometry} shows that
$\mu\nabla g(y)\in N_{\C}(y)$, proving the stationarity claim.

Finally $\dist(x,\C)\leq\norm{x-y}=r/\beta$. For any nearest point $p$,
the triangle inequality gives
$\norm{p-y}\leq\norm{p-x}+\norm{x-y}\leq2r/\beta$.
Lemma~\ref{lem:geometry} gives uniqueness when $r<\beta R$.
\end{proof}

\subsection{Inexact Proximal Descent}

\begin{lemma}[Average penalty residual]
\label{lem:nc-descent}
Suppose every inner run returns $x_{t+1}\in\B$ such that
\begin{equation}
    H_t(x_{t+1})-H_t(y_t)\leq e,\qquad
    H_t(z):=\Psi_\lambda(z)+\frac\beta2\norm{z-x_t}^2,\qquad
    y_t=Q_{\lambda,\beta}(x_t).
    \label{eq:nc-inner-gap}
\end{equation}
Then we have  
\begin{equation}
    \frac1T\sum_{t=0}^{T-1}\mathcal R_{\lambda,\beta}(x_t)^2
    \leq \frac{4\beta\Delta_\lambda}{T}
         +4\beta\left(1+\frac\beta m\right)e. 
\end{equation}
\end{lemma}

\begin{proof}
Strong convexity and \eqref{eq:nc-inner-gap} imply
$\norm{x_{t+1}-y_t}^2\leq2e/m$. Since
$H_t(y_t)\leq H_t(x_t)=\Psi_\lambda(x_t)$, we also have
\[
    \Psi_\lambda(x_{t+1})+\frac\beta2\norm{x_{t+1}-x_t}^2
       \leq\Psi_\lambda(x_t)+e.
\]
Sum this inequality and use the definition of $\Delta_\lambda$ to obtain
\[
    \frac1T\sum_{t=0}^{T-1}\norm{x_{t+1}-x_t}^2
        \leq\frac{2\Delta_\lambda}{\beta T}+\frac{2e}{\beta}.
\]
Finally,
\[
    \mathcal R_{\lambda,\beta}(x_t)^2
       \leq2\beta^2\norm{x_t-x_{t+1}}^2
              +2\beta^2\norm{x_{t+1}-y_t}^2.
\]
Averaging concludes the proof.
\end{proof}

\subsection{Proof of Theorem~\ref{thm:nonconvex}}
\label{app:nc-schedule}

Choose
\begin{equation}
\begin{split}
    T&:=\max\left\{2,\left\lceil
                        \frac{48\beta\Delta_\lambda}{\epsilon^2}
                           \right\rceil\right\},\qquad
    \delta_{\rm in}:=\frac1{6T},\\
    e_*&:=\frac{\epsilon^2}{48\beta(1+\beta/m)}.
\end{split}
\label{eq:nc-schedule}
\end{equation}
Take $K\geq2$ such that $U_K(m,G_H,\delta_{\rm in})\leq e_*$.
Doubling $K$ from $2$ until this scalar inequality holds is sufficient and
gives
\begin{equation}
    K=\widetilde O\left(
       1+\left(m+\frac{G_H^2}{m}\right)
             \frac{\beta(1+\beta/m)}{\epsilon^2}\right).
    \label{eq:nc-K-explicit}
\end{equation}

The actual inner update is an unbiased stochastic subgradient step for
$H_t$ on $\B$, with norm bounded by
$G_f+\lambda G_g+2\beta=G_H$. At $g=0$, its chosen hinge coefficient
zero is admissible. Lemma~\ref{lem:inner}, conditional on the past of each
run, and a union bound give \eqref{eq:nc-inner-gap} for all $t$ with
probability at least $5/6$. On this event Lemma~\ref{lem:nc-descent} and
\eqref{eq:nc-schedule} yield
\[
    \frac1T\sum_{t=0}^{T-1}\mathcal R_{\lambda,\beta}(x_t)^2
        \leq\frac{\epsilon^2}{12}+\frac{\epsilon^2}{12}
        =\frac{\epsilon^2}{6}.
\]
The independent uniform index $J\in\{0,\ldots,T-1\}$ therefore satisfies
$\mathcal R_{\lambda,\beta}(x_J)\leq\epsilon$ with conditional probability
at least $5/6$. As before, the unconditional success probability is at
least $25/36>2/3$.
The strict target restrictions $\epsilon<h_{\rm feas}$ and
$\epsilon<\beta R$ allow us to apply Lemma~\ref{lem:nc-transfer},
which gives \eqref{eq:nc-guarantee} and unique terminal projection.
Regardless of this event, $\C$ is compact and nonempty, so any exact
nearest point returned by the oracle is feasible. This proves
Theorem~\ref{thm:nonconvex}.

A known upper bound on $\Delta_\lambda$ can replace it in the schedule;
in particular $\Delta_\lambda\leq2G_\lambda$. No proximal residual or
distance to $\C$ is evaluated during the iterations.
The guarantee is a nearby-stationarity certificate, not a claim that the
unknown witness $y$ or its multiplier is explicitly returned.

\subsection{Why Boundary Regularity Alone Does Not Suffice}
\label{app:obstruction}

Consider the one-dimensional problem with $f(x)=0$, $\B=[-1,1]$, and
\[
    g(x)=\left(x^2-\frac14\right)^2-\frac1{64}.
\]
Its feasible set is
$\C=\{x:1/8\leq x^2\leq3/8\}\subset(-1,1)$, with two components.
The gradient is Lipschitz on a neighborhood of $\B$, and it is nonzero
at all four boundary points. Nevertheless,
$g(0)=3/64>0$ and $g'(0)=0$.
Starting Algorithm~\ref{alg:one-projection} at $x_0=0$ with the exact
zero objective oracle leaves every inner and outer iterate at zero.
Indeed, if $\beta>\ell+\lambda L_g$, the penalized proximal objective is
strongly convex and has derivative zero at zero, so its exact minimizer
is also zero. Consequently the penalty residual is zero there, but its
proximal witness is infeasible.

This example does not show that every first-order method fails, nor does
it say that the final projected point in this particular constant-objective
problem is nonstationary. It shows precisely why boundary regularity cannot
justify the implication from a small penalty residual to a feasible proximal
witness. The global condition \eqref{eq:infeasible-slope} excludes this
obstruction. It is a substantive restriction, especially for disconnected
or complicated nonconvex feasible regions.

\paragraph{Why not require every proximal subproblem to be feasible?}
For nonconvex $g$, requiring every penalized proximal subproblem to have a
feasible minimizer can couple the penalty coefficient to $\beta$, while
strong convexity simultaneously requires $\beta>\ell+\lambda L_g$.
Those demands need not be compatible. Our proof imposes no such
per-subproblem exactness requirement. It first fixes
$\lambda>G_f/\sigma_{\rm out}$, then fixes $\beta>\ell+\lambda L_g$,
and only forces feasibility of a proximal witness once its residual is
below $h_{\rm feas}$. No compatibility condition between an inner
exact-penalty margin and the proximal curvature is needed.

\section{Additional Experimental Details}
\label{app:experimental-details}

\subsection{Transfer Learning}
\label{app:transfer-details}
\paragraph{Data and objective.}
Table~\ref{tab:transfer-datasets} lists the source and target training
sample sizes and model architectures. For a labeled sample $D$ and a
real-valued score $s_\theta$, the empirical logistic risk is
\[
    R_D(\theta)=\frac{1}{|D|}\sum_{(x,y)\in D}
       \log\bigl(1+\exp(-(2y-1)s_\theta(x))\bigr),
       \qquad y\in\{0,1\}.
\]
We minimize $R_T$ subject to $R_S\leq\tau_S$, with
$\tau_S=1/\sqrt{n_S}$. This is a cap on the raw empirical source risk,
not an excess-risk constraint of the form
$R_S\leq\inf_\theta R_S(\theta)+\tau_S$; no assumption that the minimum
source logistic risk equals zero is needed to describe these experiments.

\begin{table}[!htbp]
\centering
\caption{Transfer-learning data and models. Counts are class 0/class 1
in the target and source training samples.}
\label{tab:transfer-datasets}
\small
\setlength{\tabcolsep}{5pt}
\renewcommand{\arraystretch}{1.15}
\begin{tabular}{@{}lcccl@{}}
\toprule
Dataset & Target & Source & Variables & Model \\
\midrule
Gaussian & 40/40 & 1000/1000 & 10 & Linear with intercept \\
ClimSim rainfall & 100/100 & 950/950 & 124 & $[8]$ Softplus MLP \\
\bottomrule
\end{tabular}
\end{table}

\paragraph{Budgets and reporting.}
Both methods are evaluated over ten paired test seeds. The numbers of
constrained target updates differ across methods and are reported in
Table~\ref{tab:transfer-optimization-details}; they exclude the
source-reference fitting stage. The recorded runtime includes both
source-reference fitting and constrained target optimization. Speedups
are ratios of mean runtimes, not means of per-seed ratios. Figure~\ref{fig:transfer-training-runtime}
in the main text summarizes these runtime comparisons. The ClimSim final source risk rounds to $0.02295$, while the displayed
source-risk limit rounds to $0.02294$. The unrounded residual is within
the configured feasibility tolerance of $10^{-5}$; all locked runs were
therefore recorded as numerically feasible.

\begin{table}[!htbp]
\centering
\caption{Transfer-learning update budgets and final empirical source risks.
Source risks are means $\pm$ sample standard deviations over ten runs;
zeros in the reported standard deviations reflect rounding.}
\label{tab:transfer-optimization-details}
\small
\setlength{\tabcolsep}{5pt}
\renewcommand{\arraystretch}{1.15}
\begin{tabular}{@{}llrcc@{}}
\toprule
Dataset & Method & Updates & Source-risk cap & Final source risk \\
\midrule
\multirow{2}{*}{Gaussian} & One Projection & 3,000 & 0.02236 & $0.022\pm0.001$ \\
 & Projected SGD & 500 & 0.02236 & $0.022\pm0.000$ \\
\addlinespace[3pt]
\multirow{2}{*}{ClimSim} & One Projection & 1,000 & 0.02294 & $0.02295\pm0.00000$ \\
 & Projected SGD & 100 & 0.02294 & $0.02295\pm0.00000$ \\
\bottomrule
\end{tabular}
\end{table}

\paragraph{Synthetic Gaussian data.}
The synthetic experiment used 10-dimensional Gaussian observations with
identity covariance. The target distributions were
\[
    X_{T,0}\sim N(\boldsymbol{0},I),
    \qquad
    X_{T,1}\sim N(\boldsymbol{1},I),
\]
and the source distributions were
\[
    X_{S,0}\sim N(-2\boldsymbol{1},I),
    \qquad
    X_{S,1}\sim N(\boldsymbol{1},I).
\]
Thus, class 1 had the same distribution in both domains, whereas source
class 0 was shifted relative to target class 0. Each seed used 40
target-training and 1,000 source-training observations per class.
The target validation and test sets each contained 500 observations per
class. We used a linear score with an intercept.

\paragraph{ClimSim rainfall data.}
We constructed domains by grouping consecutive location identifiers into blocks
of four. The target domain comprised locations 104--107 (derived group 26),
and the source domain comprised locations 108--111 (derived group 27).
Heavy rain was defined as a rain rate above the 95th percentile of nonzero
training-period rainfall. This threshold was fixed before constructing
the validation and test labels. The data were divided chronologically,
and preprocessing was fitted using the union of the source and target
training samples only.

The experiment used deliberately balanced samples rather than the natural
rainfall prevalence. Each seed contained 100 target-training and 950
source-training observations per class, 200 target-validation observations
per class, and 400 target-test observations per class. The model was a
one-hidden-layer $[8]$ Softplus MLP receiving the 124 climate variables.

\paragraph{Optimization settings.}
For Gaussian data, the shared source-reference model was fitted for 1,000
updates with learning rate $0.05$. One Projection used learning rate
$0.5$, 3,000 target-gradient updates, and minibatches of size 32;
Projected SGD used learning rate $1.0$, 500 updates, and the same batch
size. The linear source-risk projection was solved as a convex numerical
projection.

For ClimSim, the shared source-reference model was fitted for 200 updates
with learning rate $0.003$. One Projection used learning rate $0.05$,
1,000 target-gradient updates, and minibatches of size 40; Projected SGD
used learning rate $0.3$, 100 updates, and the same batch size. Neural
source-risk projections used the augmented-Lagrangian solver with
feasibility tolerance $10^{-5}$. All reported timing experiments ran on
the CPU.

\FloatBarrier
\paragraph{Training curves.}
Figures~\ref{fig:gaussian-transfer-training-histories}
and~\ref{fig:climsim-transfer-training-histories} show target validation
risk, target training risk, and source risk against cumulative updates.
The two methods are shown only over their respective update budgets.

Points show means over ten runs and shaded regions show one sample
standard deviation. One Projection points are outer-loop outputs plotted
at their cumulative gradient-step counts, while Projected SGD points are
stored evaluation checkpoints. The final One Projection point includes
the terminal feasibility projection. The dotted horizontal line denotes
the maximum permitted empirical source risk.

\input{ICLR/figures/transfer_synthetic_gaussian}
\ifdefined\arxivexpandedexperiments
\else
\input{ICLR/figures/transfer_climsim}
\fi
\FloatBarrier

\subsection{Classification with Abstention}
\label{app:abstention-details}
\paragraph{Empirical problem and metrics.}
Let $\{(x_i,y_i)\}_{i=1}^{n}$ be the training sample,
$n_j=|\{i:y_i=j\}|$, and $\theta=(\theta_0,\theta_1)$.
With $\phi(u)=[1+u]_+$, we use
\begin{align*}
    \widehat R_{\mathrm{amb},\phi}(\theta)
       &=\frac1n\sum_{i=1}^{n}
          \phi(s_0(x_i;\theta_0))\phi(s_1(x_i;\theta_1)),\\
    \widehat R_{\phi,j}(\theta_j)
       &=\frac1{n_j}\sum_{i:y_i=j}\phi(-s_j(x_i;\theta_j)),
          \qquad j\in\{0,1\}.
\end{align*}
The empirical feasible set requires
$\widehat R_{\phi,j}(\theta_j)\leq\beta_j+\epsilon_j$ for both
classes, retaining the calibration slacks described below.
On the test sample, class-$j$ error is the fraction of class-$j$
observations for which $s_j<0$, and ambiguity is the fraction for which
both scores are nonnegative. These indicator metrics differ from the
surrogate quantities optimized during training.

\paragraph{Calibration and evaluation.}
All abstention experiments use 500 parameter updates. Model architecture,
feature representation, learning rate, and surrogate thresholds
$(\beta_0,\beta_1)$ were selected using validation data. The thresholds
were calibrated separately for each optimizer to give comparable error
and ambiguity rates, then held fixed for evaluation over ten test seeds.
Test data were not used for model selection or hyperparameter tuning.
Table~\ref{tab:appendix-d-datasets} gives the nominal error targets
$\alpha_j$, and Table~\ref{tab:abstention-calibration} reports the chosen
surrogate thresholds and accuracy among decided observations.
Here $\beta_j$ denotes a surrogate threshold, not the proximal
regularization parameter in Algorithm~\ref{alg:one-projection}.

\begin{table}[!htbp]
\centering
\caption{Abstention datasets and nominal class-conditional error targets
$\alpha_0=\alpha_1=\alpha$. Counts refer to the full experimental samples;
architectures and feature processing are described in the text.}
\label{tab:appendix-d-datasets}
\small
\setlength{\tabcolsep}{6pt}
\renewcommand{\arraystretch}{1.15}
\begin{tabular}{@{}lrrrc@{}}
\toprule
Dataset & Observations & Class 0/class 1 & Raw variables & $\alpha$ \\
\midrule
Gaussian & 2,000 & 1,000/1,000 & 8 & 0.15 \\
ClimSim rainfall & 11,373 & 10,825/548 & 124 & 0.05 \\
NASA rainfall & 49,302 & 47,073/2,229 & 6 & 0.10 \\
Credit default & 23,864 & 21,236/2,628 & 11 & 0.20 \\
\bottomrule
\end{tabular}
\end{table}

The surrogate thresholds $\beta_j$ were calibrated on validation data
to move the class-conditional errors toward the targets $\alpha_j$.
Training enforced
\[
    \widehat R_{\phi,j}(\theta_j)
    \leq \beta_j+\epsilon_j,
    \qquad
    \epsilon_j=\frac{C}{\sqrt{n_j}}.
\]
The shared slack constant $C$ was selected using validation data and
fixed at $C=0.5$ for the reported experiments. Satisfying this surrogate
constraint does not guarantee that the held-out class-conditional error
is at most $\alpha_j$. In particular,
both methods have about $9\%$ class-0 test error on ClimSim despite the
nominal $5\%$ target; the class-1 errors on credit default are also
slightly above the nominal $20\%$ target. Moreover, the different
calibrated thresholds used by the two optimizers on ClimSim and NASA
mean that they are compared at similar predictive tradeoffs rather than
on identical constrained problems.

\begin{table}[!htbp]
\centering
\caption{Abstention calibration and supplementary test accuracy.
Thresholds are selected on validation data; accuracy is reported as
mean $\pm$ sample standard deviation over ten runs.}
\label{tab:abstention-calibration}
\small
\setlength{\tabcolsep}{8pt}
\renewcommand{\arraystretch}{1.12}
\begin{tabular}{@{}llcc@{}}
\toprule
Dataset & Method & $(\beta_0,\beta_1)$ & Decided accuracy (\%) \\
\midrule
\multirow{2}{*}{Gaussian} & One Projection & $(0.300,0.300)$ & $84.95\pm2.14$ \\
 & Projected SGD & $(0.300,0.300)$ & $84.77\pm1.96$ \\
\addlinespace[3pt]
\multirow{2}{*}{ClimSim} & One Projection & $(0.180,0.290)$ & $91.69\pm0.60$ \\
 & Projected SGD & $(0.140,0.225)$ & $91.96\pm0.98$ \\
\addlinespace[3pt]
\multirow{2}{*}{NASA rainfall} & One Projection & $(0.260,0.210)$ & $89.35\pm0.39$ \\
 & Projected SGD & $(0.260,0.280)$ & $89.46\pm0.44$ \\
\addlinespace[3pt]
\multirow{2}{*}{Credit default} & One Projection & $(0.425,0.400)$ & $79.77\pm1.73$ \\
 & Projected SGD & $(0.425,0.400)$ & $78.78\pm1.58$ \\
\bottomrule
\end{tabular}
\end{table}

\paragraph{Numerical implementation.}
For the One Projection updates, the two residuals are combined as
\[
    g_{\max}(\theta)=\max_{j\in\{0,1\}}
       \{\widehat R_{\phi,j}(\theta_j)-\beta_j-\epsilon_j\}.
\]
The projection routines enforce the two inequalities individually.
For linear scores, Euclidean projection onto the empirical feasible set
is a convex problem and is solved to numerical tolerance. For neural
networks, an augmented-Lagrangian solver approximately minimizes
$\frac12\|\theta-\bar\theta\|^2$ subject to the empirical constraints.
It seeks a feasible point near $\bar\theta$ but does not certify a globally
nearest feasible point. One Projection uses independent minibatches to estimate the ambiguity
objective and each class-conditional constraint gradient. Projected SGD
uses an objective minibatch followed by a full-training-set feasibility
projection after every parameter update.
Thus these experiments evaluate a practical implementation: the hinge
and maximum constraints can be nonsmooth, constraint estimates are
stochastic, and the neural-network projection is approximate. The exact
constraint-oracle and projection assumptions of the nonconvex-constraint
theorem are not verified by these experiments.

\paragraph{Synthetic Gaussian.}
For each seed, we generated 2,000 observations from two balanced
eight-variable Gaussian classes, with class separation 2.0,
identity-scaled covariance, and 2\% label noise. The split contains
1,200 training, 400 validation, and 400 test observations. We use two
independent linear scores with intercepts.

\paragraph{ClimSim rainfall.}
The 11,373 observations were split chronologically into 6,844 training,
2,257 validation, and 2,272 test observations. Class 1 denotes a rain
rate above the training-set 95th percentile. Each score is an independent
$[32,16]$ Softplus MLP with 248 inputs: 124 climate variables and their
squares.

\paragraph{NASA rainfall.}
The 49,302 observations were divided using a seeded, stratified
60/20/20 random split: 29,581 training, 9,860 validation, and 9,861 test
observations. The available data artifact contains no dates, so the
split is not chronological. Each score is an independent $[8]$ Softplus
MLP with 27 degree-two features derived from six NASA variables.

\paragraph{Credit default.}
The 23,864 labeled observations were divided using a seeded, stratified
60/20/20 random split: 14,318 training, 4,772 validation, and 4,774 test
observations. Default is class 1. Each score is an independent $[16,8]$
ReLU MLP with 11 inputs after a Yeo--Johnson transformation fitted on
the training data.

\paragraph{Optimization and timing details.}
All experiments used minibatches of size 64 and 500 parameter updates.
One Projection used \(T=20\) outer iterations and \(K=25\) inner
iterations, with proximal coefficient \(\ell=0.1\), penalty strength
\(5.0\), and the last outer iterate. Its joint parameter-ball radius was
2 for the linear Gaussian model and 10 for the MLP models. The selected
One Projection/Projected SGD learning rates were \(0.05/0.05\) for
Synthetic Gaussian, \(0.05/0.05\) for ClimSim, \(0.05/0.03\) for NASA,
and \(0.10/0.03\) for credit default.

The linear Gaussian experiment used the exact empirical-risk projector
with feasibility tolerance \(10^{-5}\), at most 30 bisection steps, and
at most 500 dual steps. The MLP experiments used the approximate
augmented-Lagrangian projector with learning rate \(0.01\), 150 inner
iterations per multiplier round, six multiplier rounds, initial penalty
1, penalty growth factor 5, maximum penalty \(10^6\), 25 feasible-refinement
iterations, feasibility tolerance \(10^{-5}\), and boundary-refinement
tolerance \(10^{-4}\).

Runs were deterministic and CPU-only, using an Intel Core i7-7820X,
Python 3.11.9, and PyTorch 2.13.0. Reported training runtime begins with
the optimizer's parameter updates and ends after its final training
projection. It includes all bounding-set and feasible-set projections,
but excludes data loading and preprocessing, feasible initialization,
post-training evaluation, figure generation, and artifact serialization.

\FloatBarrier
\paragraph{Additional training curves.}
The ClimSim curves appear in Figure~\ref{fig:climsim-rainfall-training-histories}
in the main text. Figures~\ref{fig:synthetic-gaussian-training-histories},
\ref{fig:nasa-rainfall-training-histories}, and
\ref{fig:credit-default-training-histories} below give the corresponding
results for Gaussian, NASA rainfall, and credit default data.
All four figures show the product surrogate objective and class-conditional
surrogate risks. Points and colored bands report the mean and one sample
standard deviation across ten runs. One Projection checkpoints are
outer-loop outputs, with the last point evaluated after the terminal
projection; Projected SGD is recorded at the corresponding cumulative
update counts. Dotted lines show the validation-calibrated thresholds
$\beta_j$, separately for each optimizer where these differ. Gray bands
show $\beta_j\pm\epsilon_j$ with $\epsilon_j=0.5/\sqrt{n_j}$; they are
reference ranges, not confidence intervals. The enforced empirical
constraint is $\widehat R_{\phi,j}\leq\beta_j+\epsilon_j$.
One Projection permits intermediate constraint violation and applies
the feasible-set projection only at the end.

\input{ICLR/figures/abstention_synthetic_gaussian}
\input{ICLR/figures/abstention_nasa}
\input{ICLR/figures/abstention_credit_default}
\FloatBarrier

\end{document}